\documentclass[11pt, reqno]{amsart}

\usepackage{color}
\usepackage{amsmath}
\usepackage{latexsym,amssymb}
\usepackage[mathscr]{euscript}
\usepackage{stmaryrd}
\usepackage{bm}
\usepackage{enumerate}
\usepackage{setspace}
\newtheorem{thm}{Theorem}[section]
\newtheorem{proposition}{Proposition}[section]

\newtheorem{corollary}{\bf Corollary}[section]
\newtheorem{remark}{\bf Remark}[section]
\newtheorem{lemma}{\bf Lemma}[section]
\newtheorem{definition}{Definition}[section]
\numberwithin{equation}{section}
\newtheorem{assumption}{Assumption}[section]

\begin{document}
	
	\baselineskip=17pt
	
	\title[]
	{ Zero-Sum Games for piecewise deterministic Markov decision processes with risk-sensitive finite-horizon cost criterion}
	
	\author[Subrata Golui]{Subrata Golui}
	\address{Department of Mathematics\\
		Indian Institute of Technology Bombay\\
		Mumbai, Maharashtra 400076, India}
	\email{golui@math.iitb.ac.in}
	
%
	
	
	\date{}
	
	\begin{abstract}
		\vspace{2mm}
		\noindent
		This paper investigates the two-person zero-sum stochastic games for piecewise deterministic Markov decision processes with risk-sensitive finite-horizon cost criterion on a general state space. Here, the transition and cost/reward rates are allowed to be unbounded from below and above.  Under some mild conditions, we show the existence of the value of the game and an optimal randomized Markov saddle-point equilibrium in the class of all admissible feedback strategies. By studying the corresponding risk-sensitive finite-horizon optimal differential equations out of a class of possibly unbounded functions, to which the extended Feynman-Kac formula is also justified to hold, we obtain our required results.
		
		\vspace{2mm}
		
		
	\end{abstract}
	
	\maketitle
	
\section{INTRODUCTION}
This paper analyzes the two-person risk-sensitive finite-horizon zero-sum stochastic games  for piecewise deterministic Markov decision processes (PDMDPs) with unbounded transition and cost/reward rates under admissible strategies. Here the state space is a general Borel set and the action spaces are Borel sets. Our goal is to prove the existence of value of the game and saddle point equilibrium and give a characterization of the optimal strategies in terms of the corresponding risk-sensitive finite-horizon optimality equation (Shapley equation).
It is well known that, game theory is a mathematical framework to analyze real-life situations among competing players and produce satisfactory decision-making for competing players. Game theory has a wide range of applications such as psychology, evolutionary biology, politics, social sciences, economics and business. Markov games are one type of stochastic dynamic games, where the state dynamics of the games are determined by Markov processes. Now a days, Markov games have received increasing attentions and have been widely investigated; see, for instance, \cite{HLG_MMOR_2023}, \cite{GH4},  \cite{GH1},  \cite{S_PMS_2020}. 

This paper deals with zero-sum piecewise deterministic Markov games (PDMGs), one type of Markov games whose state dynamics are driven by piecewise deterministic Markov decision processes (PDMDPs). PDMDPs evolve through random jumps at random time points while the motion between jumps follows a flow. In particular, if the flow remains unchanged over time, PDMDPs reduce to continuous-time (pure jump) Markov processes (CTMPs). Roughly speaking, the uncontrolled version of the process evolves as follows: given the current state, the process evolves deterministically according to some flow mapping $\phi$, up to the next jump, taking place after a random time whose distribution is (nonstationary) exponential, and the dynamics continue in the similar manner. A detailed book treatment with many examples of this and more general type of processes, allowing deterministic jumps, can be found in \cite{D_1993}. These features of PDMDPs make them have wide applications in many areas such as management science, operations research, and engineering. There has been a vast literature on PDMDPs, \cite{BR_SPR_2011}, \cite{CDP_SIAM_2016}, \cite{HG_ST_2019}. However, there are very few literature that considers piecewise deterministic Markov games (PDMGs), \cite{CD_AMO_2018}.
 
Risk-sensitive game is a generalization of classical stochastic game in which the degree of risk aversion or risk tolerance of the players in the games is explicitly parameterized in the cost criterion and influences the outcome of the game directly, \cite{GP}. Risk-sensitive or ``exponential of integral" cost is a popular cost criterion, particularly in finance (see, e.g., \cite{BP}) because it captures the effects of higher order moments of the cost in addition to its expectation. The  risk-sensitive  cost has many important  applications including  mathematical  finance see, Whittle \cite{WH} and references therein.  In the literature, the risk-neutral criterion is studied extensively for CTMDPs  under different sets of conditions \cite{GH2}. In risk-neutral stochastic games, players ignore the risk since they usually consider the expectation of the integral of costs. This along the lines of the expected utility theory of von Neumann and Morgenstern. Risk-neutral criterion has some limitations namely if the variance is large then there can be issues with the optimal control. In real life, the different decision-makers may have the different risk preferences. Thus, it is desirable to take the risk preferences of the decision-makers into consideration in the performance criterion. The risk-sensitive finite-horizon cost criterion uses the exponential utility function to reflect the risk preferences of the players. In 1995, Bell \cite{B} gave a model to quantify the trade off between risk and return through the notion risk-return function and came up with classes of utility function which have risk-return interpretation and one of them is exponential utility function. Finite-horizon risk-sensitive CTMDP is considered in  \cite{GS}, \cite{GLZ},  \cite{GZ}, \cite{W}, while for infinite horizon risk-sensitive CTMDP see,  \cite{GS}, \cite{KP2}, \cite{KP1}, \cite{Z1}, and the references therein. Infinite horizon risk-sensitive CTMDP for piecewise deterministic Markov decision processes has been studied in \cite{HLG_OR_2022}, \cite{GZ2}.  In \cite{GLZ}, the authors investigated the finite-horizon risk-sensitive continuous-time Markov decision processes with unbounded transition and cost rates on a countable state space. 

Although there are some literature for risk-sensitive control of PDMDPs, \cite{HLG_ORL_2018}, \cite{HLG_OR_2022}, \cite{GZ2} but the corresponding literature in the context of risk-sentive PDMGs is not investigated yet.  This is the first study that investigates risk-sensitive PDMGs.



Here, we consider the risk-sensitive finite-horizon zero-sum games for piecewise deterministic Markov decision processes with the unbounded costs  and transition rates in the class of all admissible multi-strategies on general state space, which have not been studied yet. This paper makes an attempt to study the existence of saddle point equilibria under the risk-sensitive finite-horizon cost criterion for PDMDPs. The game model in this paper is nonhomogeneous (i.e., the transition rates and costs depend on the time parameter).  In section 3, we show that the value function has an upper and lower bound. Then we state the extended Feynman-Kac formula. Here, we also establish the comparison theorem. In section 4, we show that if the cost and transiton rates are bounded then the Shapley equation has a solution. Proposition \ref{prop 4.2}, conveys that the value function is decreasing in $t$. Under some suitable assumptions, in Theorem \ref{thm 4.1}, we establish that if the cost rates are nonnegative, the optimality equation has a solution. Atlast in Theorem \ref{thm 4.2}, we prove our final results, i.e., we prove that if the cost rates are real-valued functions, then the Shapley equation (\ref{eq 3.13}), has a solution  and by measurable selection theorem in \cite{N}, we show that the existence of saddle point equilibria. Then by Feynman-Kac formula, we ensure that the solution of the Shapley equation is unique. We also show that the value of the game exists.
The rest of this article is organized as follows. Section 2 provides the game model, some definitions, and the finite-horizon zero-sum game criterion. In Section 3, we discuss some preliminary facts, some assumptions, the Feynman-Kac formula and  prove the stochastic representation of the solution of the Shapley equation (\ref{eq 3.13}). The existence of unique solution to the risk-sensitive finite-horizon Shapley equation, the value of the game and saddle-point equilibrium in the class of Markov strategies for the  zero-sum game are proven in Section 4.
\section{The game model}
In this section we introduce the two-person zero-sum stochastic game  model for piecewise deterministic Markov decision processes comprises the following elements
\begin{equation}
\mathbb{M}:=\{S, A, B, (A(t,x)\subset A, B(t,x)\subset B, t\geq 0, x\in S),q( \cdot|t,x, a,b),\phi(x,t), c(t,x, a,b), g(t,x)\} \label{2.1}
\end{equation}
\begin{itemize}
	\item  a Borel space $S$ (Borel subset of a complete seperable metric space), called the state space, whose elements are referred to as states of the system and the corresponding Borel $\sigma$-algebra is $\mathcal{B}(S)$. 
	\item   $A$ and $B$ are the action sets for players 1 and 2, respectively. The action spaces $A$ and $B$ are assumed to be Borel spaces with the Borel $\sigma$-algebras $\mathcal{B}(A)$ and $\mathcal{B}(B)$, respectively.
	\item  For each $t\geq 0$ and $x\in S$, $A(t,x)\in \mathcal{B}(A)$ and $B(t,x)\in \mathcal{B}(B)$ denote the sets of admissible actions for players 1 and 2 in state $x$, respectively. Let $K:=\{(t,x, a,b)|x\in S, a\in A(t,x), b\in B(t,x)\}$, which is a Borel subset of $[0,\infty)\times S\times A\times B$.
	
	\item Given any $(t,x, a,b)\in K$, the transition rate $q(\cdot |t, x, a,b)$ is a signed kernel on $S$ such that  $q(D| t,x, a,b)\geq 0 $ where $(t,x,a,b)\in K$ and $x\notin D$. Moreover, we assume that $q(\cdot | t,x, a,b)$ is conservative in the sense of $q(S|t,x, a,b)\equiv 0$, and stable in that of
	\begin{equation}
	q^{*}(x):=\sup_{t\geq 0,a\in A(t,x),b\in B(t,x)}q(t,x, a,b)<\infty \quad \forall x \in S,\label{eq 2.2}	\end{equation} 
	where $q(t,x, a,b):=-q(\{x\} |t, x, a,b)\geq 0 \text{ for all } (t,x, a,b)\in K.$
	\item The function $\phi(x,t)$ is called a flow, is a measurable function from $S\times \mathbb{R}$ to $S$. The flow defines the motions between jumps of the PDMDPs. We assume that $\phi(x,s+t)=\phi(\phi(x,s),t)$ for all $x\in S$ and $(s,t)\in \mathbb{R}^2$.  $\phi(x,0)=x$, the function $ t \rightarrow \phi(x, t)$ is continuous.
	\item Finally, the real-valued running cost rate function $c$ is measurable on $K$ and the real-valued terminal cost function $g$ is measurable on $[0,\infty)\times S$. 
\end{itemize}


The goal of player 1 is to maximize his/her rewards, whereas that of player 2 is to minimize his/her costs with respect to some performance criterion $\mathscr{J}(\cdot,\cdot, \cdot,\cdot)$, which in our present case is defined by (\ref{eq 2.5}), below.
To formalize what is described above, below we describe the construction of piecewise deterministic Markov decision processes (PDMDPs) under possibly admissible strategies.
To construct the underlying PDMDPs (as in \cite{CDP_SIAM_2016}, \cite{HLG_ORL_2018}, \cite{GZ_AMP_2018}, \cite{HG_ST_2019}), we introduce some notations: let $S_\Delta:=S \cup \{\Delta\}$ (with some isolated state $\Delta \notin S$), $\Omega^0:=(S\times(0,\infty))^\infty$, $\Omega:=\Omega^0 \cup \{(x_0, \theta_1, x_1, \cdots ,\theta_k, x_k, \infty, \Delta, \infty,\Delta, \cdots)|x_0 \in S, x_l \in S, \theta_l\in (0,\infty),$ \text{for each} $1\leq l\leq k, k\geq 1\}$, and let $\mathscr{F}$ be the Borel $\sigma$-algebra on $\Omega$. Then we obtain the measurable space $(\Omega, \mathscr{F})$. 
For each $k\geq 0$, $ \omega:=(x_0, \theta_1, x_1, \cdots , \theta_k, x_k, \cdots)\in \Omega,$ define $X_0(\omega):=x_0$, $T_0(\omega):=0$, $X_k(w):=x_k$, $T_k(\omega):= T_{k-1}(\omega)+\theta_{k}$, $T_\infty(\omega):=\lim_{k\rightarrow\infty}T_k(\omega)$. Using $\{T_k\}$, we define the state process $\{\xi_t\}_{t\geq 0}$ as
	\begin{align}
\xi_t(\omega):=\left\{ \begin{array}{llll}&\phi(X_n(\omega),t-T_n(\omega)),~\text{if}~T_n(\omega)\leq t<T_{n+1}(\omega),\\&\Delta,~\text{if}~t\geq T_{\infty}.\label{eq 2.3}
\end{array}\right.
\end{align}
The process after $T_\infty$ is regarded to be absorbed in the state $\Delta$. Thus, let $q(\cdot | t,\Delta, a_\Delta,b_\Delta):\equiv 0$, $A_\Delta:=A\cup \{a_\Delta\}$, $B_\Delta:=B\cup \{b_\Delta\}$, $ A(t,\Delta):=\{a_\Delta\}$, $B(t,\Delta):=\{b_\Delta\}$, $c(t,\Delta, a,b):\equiv 0$ for all $(a,b)\in A_\Delta\times B_\Delta$, where $a_\Delta$, $b_\Delta$ are isolated points. Moreover, let $\mathscr{F}_t:=\sigma(\{T_k\leq s,X_{k}\in D\}:D\in \mathcal{B}(S),0\leq s\leq t, k\geq0)$ for all $t\geq 0$, $\mathscr{F}_{s-}=:\bigvee_{0\leq t<s}\mathscr{F}_t$, and $\mathscr{P}:=\sigma(\{A\times \{0\},A\in \mathscr{F}_0\} \cup \{ B\times (s,\infty),B\in \mathscr{F}_{s-}\})$ which denotes the $\sigma$-algebra of predictable sets on $\Omega\times [0,\infty)$ related to $\{\mathscr{F}_t\}_{t\geq 0}$.
In order to define the risk-sensitive cost criterion, we need to introduce the definition of strategy below.
\begin{definition}
	An admissible feedback strategy for player 1, denoted by $	\pi^1=\{\pi^1(t)\}_{t\geq 0}$, is a transition probability $	\pi^1(da| \omega, t)$ from $(\Omega\times[0,\infty),\mathscr{P})$ onto $(A_\Delta,\mathcal{B}(A_\Delta))$, such that $	\pi^1(A(\xi_{t-}(\omega))| \omega, t) = 1$.
	Using appropriate projections of the transition kernel $\pi^1$, an admissible feedback strategy for player 1, determines and is, in turn,  determined by a sequence $\{\pi _k^1,k\geq 0\}$ of stochastic kernel on $A$ such that
	\begin{align*}
	\pi^1(da | \omega,t)&=I_{\{t=0\}}(t)\pi _0^1(da|x^{'}_0, 0)+\sum_{k\geq 0}I_{\{T_k< t\leq T_{k+1}\}}\pi^1 _k(da|x^{'}_0, \theta_1, x^{'}_1, \dots , \theta_k, x^{'}_k, t-T_k)\notag\\
	&+I_{\{t\geq  T_\infty\}}\delta_{a_\Delta}(da),
	\end{align*}
	where $\pi _0^1(da|x^{'}_0, 0)$ is a stochastic kernel on $A$ given $S$ such that $\pi _0^1(A(x^{'}_0)|x^{'}_0, 0)=1$, $\pi^1 _k (k\geq 1)$ are stochastic kernels on $A$ given $(S\times (0,\infty))^{k+1}$ such that $\pi _k^1(A(x^{'}_k)|x^{'}_0,\theta_1,x^{'}_1,\cdots,\theta_k,x^{'}_k,t-T_k )=1$, and $\delta_{a_\Delta}(da)$ denotes the Dirac measure at the point $a_\Delta$.
\end{definition} For more details see \cite[Definition~2.1, Remark~2.2]{GS1}, \cite{PZ}, \cite{Z_JAP_2014}.
The set of all admissible feedback strategies for player 1 is denoted by $\Pi^1$.
A strategy $\pi^1\in \Pi^1$ for player 1, is called a Markov if $\pi^1(da | \omega,t)=\pi^1(da | \xi_{t-}(w),t)$ for every $w\in \Omega$ and $t\geq 0$, where $\xi_{t-}(w):=\lim_{s\uparrow t}\xi_s(w)$. We denote by  $\Pi_{m}^1$ the family of all Markov strategies, for player 1. The  sets of  admissible feedback strategies $\Pi^2$,  Markov strategies $\Pi_m^2$ for player 2 are defined analogously.

For each $(\pi^1,\pi^2)\in \Pi^1\times \Pi^2$, the random measure $m^{\pi^1,\pi^2}$ defined by 
	\begin{equation}
		m^{\pi^1,\pi^2}(B|\omega,t)dt:=\int_{B}\int_{A}q(B|t,\xi_{t-},a,b)\pi^1(da|\omega,t)\pi^2(db|\omega,t)I_{\{\xi_{t-}\notin B\}}dt,~B\in\mathcal{B}(S)\label{eq 2.4}
	\end{equation}
is predictable, see Jacod (1975) \cite{J}.
Under Assumption \ref{assm 3.1}, below, for any initial distribution $\gamma$ on $S$ and any multi-strategy $(\pi^1,\pi^2)\in \Pi^1\times \Pi^2$, Theorem 4.27 in \cite{KR} yields the existence of a unique probability measure denoted by $P^{\pi^1,\pi^2}_\gamma$ (depending on $\gamma$ and $(\pi^1,\pi^2)$) on $(\Omega,\mathscr{F})$ such that $P^{\pi^1,\pi^2}_\gamma(\xi_0=x)=1,$ and with respect to which, $m^{\pi^1,\pi^2}(\cdot|\omega,t)dt$ is the dual predictable projection of the random measure $\sum_{n\geq 1}\delta _{(T_n,X_n)}(dt,dx)$ of the marked point process $\{T_n,X_n\}$ on $\mathcal{B}((0,\infty)\times S)$, see \cite{K} or Chapter 4 of \cite{KR} for more details. Here $\delta _{(T_n,X_n)}(dt,dx)$ is the Dirac measure concentrated at $(T_n,X_n)$. Let $E^{\pi^1,\pi^2}_\gamma$ be the corresponding expectation operator. In particular, $P^{\pi^1,\pi^2}_\gamma$ and $E^{\pi^1,\pi^2}_\gamma$ will be respectively written as $P^{\pi^1,\pi^2}_x$ and $E^{\pi^1,\pi^2}_x$ when $\gamma$ is the Dirac measure at a state $x$ in $S$.\\
For any compact metric space $Y$, let $P(Y)$ denote the space of probability measures on $Y$ with Prohorov topology. For each $x\in  S$, $t\in [0,T]$, $\mu\in P(A(t,x))$ and $\nu\in P(B(t,x))$, the associated cost and transition rates are defined, respectively, as follows: 
$$c(t,x,\mu,\nu):=\int_{B(t,x)}\int_{A(t,x)}c(t,x,a,b)\mu(da)\nu(db),$$ 
$$q(\cdot|t,x,\mu,\nu):=\int_{B(t,x)}\int_{A(t,x)}q(\cdot|t,x,a,b)\mu(da)\nu(db).$$
 Let $E^{\pi^1,\pi^2}_x$ be the expectation operator with respect to  $P^{\pi^1,\pi^2}_x$. Now we provide the definition of the risk-sensitive finite-horizon cost criterion for zero-sum PDMGs. In the following, we fix any risk-sensitivity coefficient $\lambda\in(0,1]$ and the length of the horizon $T>0$. For each $x\in S$, $t\in [0,T]$ and any $(\pi^1,\pi^2) \in \Pi^1\times \Pi^2$, the risk-sensitive $T$-horizon cost criterion is defined by
\begin{equation}
\mathscr{J}^{^{\pi^1,\pi^2}}(0,x):=\frac{1}{\lambda}\ln  {E}^{\pi^1,\pi^2}_x\biggl[e^{\lambda\int_{0}^{T}\int_{B}\int_{A}c(t,\xi_t,a,b)\pi^1(da|\omega,t)\pi^2(db|\omega,t)dt+\lambda g(T,\xi_T)}\biggr]\label{eq 2.5}
\end{equation}
provided that the integral is well defined.
For each $(\pi^1,\pi^2) \in \Pi_m^1\times \Pi_m^2$, it is well known that $\{\xi_t,\geq 0\}$ is a Markov Process on $(\Omega,\mathscr{F},P^{\pi^1,\pi^2}_\gamma)$, and thus for each $x\in S$ and $t\in [0,T]$,
\begin{equation}
\mathscr{J}^{^{\pi^1,\pi^2}}(t,x):=\frac{1}{\lambda}\ln  {E}^{\pi^1,\pi^2}_{(t,x)}\biggl[e^{\lambda\int_{t}^{T}\int_{B}\int_{A}c(s,\xi_s,a,b)\pi^1(da|\xi_s,s)\pi^2(db|\xi_s,s)ds+\lambda g(T,\xi_T)}\biggr]\label{eq 2.6}
\end{equation}
is well defined.
We also need the following concepts. The functions on S defined as \\
$L(x):=\sup_{\pi^1\in \Pi^1}\inf_{\pi^2\in \Pi^2}\mathscr{J}^{\pi^1,\pi^2}(0,x)$ and $U(x):=\inf_{\pi^2\in \Pi^2}\sup_{\pi^1\in \Pi^1}\mathscr{J}^{\pi^1,\pi^2}(0,x)$ are called, respectively, the lower value and the upper value of the game. It is clear that 
$$L(x)\leq U(x)~ \text{for all}~x\in S.$$
\begin{definition}
If $L(x)=U(x)$ for all $x\in S$, then the common function is called the value of the game and is denoted by $\mathscr{J}^{*}(0,x)$.
\end{definition}
\begin{definition}
	Suppose that the game has a value $\mathscr{J}^{*}$. Then a strategy $\pi^{*1}$ in $\Pi^1$ is said to be optimal for player 1 if
	 $$\inf_{\pi^2\in \Pi^2}\mathscr{J}^{\pi^{*1},\pi^2}(0,x)=\mathscr{J}^{*}(0,x)~ \text{for all}~ x\in S.$$
	Similarly, $\pi^{*2}\in \Pi^2$ is optimal for player 2 if
	$$\sup_{\pi^1\in \Pi^1}\mathscr{J}^{\pi^1,\pi^{*2}}(0,x)=\mathscr{J}^{*}(0,x) ~ \text{for all}~x\in S.$$
		If $\pi^{*k}\in \Pi^k$ is optimal for player k (k=1,2), then $(\pi^{*1},\pi^{*2})$ is called a pair of optimal strategies and also called a saddle-point equilibrium.
\end{definition}
Some comments are in order.

\begin{remark}\label{R1} We now explain the significance of the risk-sensitive criterion.
	\begin{enumerate}
		\item[(i)]  Let $Y$ be a  random cost accrued over finite/infinite time horizon.
		Let the (constant) coefficient of absolute risk aversion be given by  $\theta \in \mathbb{R}$. Let $U_{\theta}$ be a utility function given by
		\begin{eqnarray*}
			U_{\theta}(x) = \left\{
			\begin{array} {lll}
				sgn(\theta) e^{\theta x},   & \rm{if}\: \theta \neq 0,\\
				x, & \rm{if }\: \theta = 0.\\
			\end{array}\right.
		\end{eqnarray*}
		Suppose the decision maker evaluates the random cost  $Y$ via $ E(U_{\theta}(Y))$.
		A certainty equivalent of $Y$ is a number $J(\theta, Y)$ such that
		$$ U_{\theta}(J(\theta, Y)) = E(U_{\theta}(Y)).$$
		Therefore for a person with the risk-sensitive factor $\theta$, paying the random cost $Y$ is tantamount to paying a deterministic cost
		$J(\theta, Y)$. It is easy to see that
		$$J(\theta, Y) = \frac{1}{\theta} \log(E e^{\theta Y}).$$
		If $\theta > 0$, then $J(\theta, Y) \geq E Y$. Thus the decision maker is paying a higher cost for being risk-averse.
		If $\theta <0$, then the decision maker is risk-seeking. Finally $\theta = 0$ corresponds to the risk-neutral case.
		
		The risk of a random quantity is also associated with its variance in the literature in economics. That is why the controller may wish to minimize both mean and variance.
		In the risk-neutral case: the decision make minimizes $E(Y).$
		In the risk-sensitive case: the decision maker seeks to minimize $J(\theta, Y) = \frac{1}{\theta}\log(E(e^{\theta Y})).$
		
		For a small value of $\theta$, by Taylor series expansion,
		$$J(\theta, Y)\approx E(Y) + \frac{\theta}{2}Var(Y).$$
		The right hand side above is a standard utility employed in a portfolio optimization problem. However, the above may not be suitable
		for games \cite{N1}.
		
		\item[(ii)] There are other non-linear risk-sensitive utility functions, e.g., power, logarithm. But these utility functions do not
		lead to certainty equivalence. Note that usually utility of a payoff is defined as a concave function. By analogy we are treating a convex function as  ``utility" associated with the cost.
		
		\item[(iii)] In case the random cost $Y$ is determined by two players who are strictly competitive, then in the risk-neutral case the sum of the expected random cost and the expected
		random payoff is zero. Thus in this case we can define saddle-point equilibrium. However, this is not going to be the case for any non-linear utility function including the one we are addressing here. Thus  risk-sensitive zero-sum games have to be studied via Nash equilibria \cite{N1}. 
		
		\item[(iv)] The  stochastic game with cost criterion (\ref{eq 2.5})  has  primarily been formulated from the viewpoint of the second player, i.e. the minimizer who is risk-averse. The first player (the maximizer)
		is a virtual player who is antagonistic   to the second player. Such games have applications in queueuing systems where each player treats the rest of the players as a superplayer
		antagonistic to him/her. We refer to  \cite{Alt}  for a zero-sum stochastic game in a flow control problem  in discrete time,  and \cite{GSK}, \cite{GhoshPradhan} for analogous problems in 
		continuous time. Such a game is also applied in \cite{BG2} for a temporal CAPM problem where each investor treats the rest of the investors as a superplayer antagonistic to him/her.

		\item[(v)] At any point of time, the accumulated history of the game includes past and present states, past sojourn times and past actions taken by players. In our game model  each player's  admissible strategies
		include only past and present states and past sojourn times. Hence  such strategies are called feedback strategies. Inclusion of history-dependent strategies is infeasible even for one player games; see Proposition 1 in \cite{Ney}.

			\item[(vi)]  For a fixed state $i \in S$, if we treat the game on the spaces of strategies  with the cost given by  (\ref{eq 2.5}), then it  becomes a zero-sum game.
		Thus we can define a saddle-point equilibrium, optimal strategies. This is what we have done in this work.

	\end{enumerate}
	
\end{remark}
 \section{Preliminaries}
 To prove the existence of a pair of optimal strategies, we need to develop some preliminary  facts about the risk-sensitive finite-horizon PDMGs. 
 Since the rates $q(dy|t,x,a,b)$ and costs $c(t,x,a,b)$ are allowed to be unbounded, we next assume some conditions for the non-explosion of $\{\xi_t,t\geq 0\}$ and finiteness of $ \mathscr{J}^{\pi^1,\pi^2} (0, x )$, which had been widely used in PDMDPs; see, for instance,
 \cite{CDP_SIAM_2016}, \cite{HLG_ORL_2018}, \cite{GZ_AMP_2018}, \cite{HG_ST_2019} and references therein.
 \begin{assumption}\label{assm 3.1}
	
There exist a function $V: S \to [1,\infty)$ and constants $\rho_1\geq 0$, $b_1\geq 0 $, $M_1\geq 1$ and $M_2\geq 1$ such that
	\begin{enumerate}
		\item [(i)] $\int_{S}V(\phi(y,t))q(dy |s, x, a,b)\leq \rho_1 V(\phi(x,t))+b_1$ for each $(s,x, a,b)\in K$, $t\in \mathbb{R}_{+}$;
		
		\item [(ii)] there exists a sequence $\{S_n,n\geq 1\}$ of measurable subsets of $S$ such that $S_n\uparrow S$, $\sup_{(t,x,a,b)\in K,x\in S_n} q(t,x,a,b)<\infty$, and $\lim_{n\rightarrow \infty}\inf_{x\notin S_n}V(\phi(x,t))=\infty$ for all $t\in \mathbb{R}_{+}$.
		
		\item[(iii)] $e^{2(T+1)|c(t,x,a,b)|}\leq M_2 V(\phi(x,T-t))$ and $e^{2(T+1)|g(t,x)|}\leq M_2 V(\phi(x,T-t))$ for each $(t,x,a,b)\in K_{[0,T]},$ where $K_{[0,T]}=\{(s,x,a,b):s\in [0,T],x\in S,a\in A(s,x), b\in B(s,x)\}$.
	\end{enumerate}
\end{assumption}
 Since logarithm is an increasing function, instead of studying $\mathscr{J}^{\pi^1,\pi^2}(0,x)$, we will consider $J^{\pi^1,\pi^2}(0,x)$ on $ [0,T]\times S\times \Pi^1\times\Pi^2$ and  $J^{\pi^1,\pi^2}(t,x)$ on $[0,T]\times S\times \Pi_m^1\times\Pi_m^2$, respectively given by
  \begin{align}
 J^{\pi^1,\pi^2}(0,x):= {E}^{\pi^1,\pi^2}_x\biggl[e^{{\lambda\int_{0}^{T}\int_{B}\int_{A}c(t,\xi_t,a,b)\pi^1(da|\omega,t)\pi^2(db|\omega,t)dt+\lambda g(T,\xi_T)}}\biggr].\label{eq 3.1}
 \end{align}
 and 
 \begin{equation}
 {J}^{\pi^1,\pi^2}(t,x):=  {E}^{\pi^1,\pi^2}_{(t,x)}\biggl[e^{\lambda\int_{t}^{T}\int_{B}\int_{A}c(s,\xi_s,a,b)\pi^1(da|\xi_s,s)\pi^2(db|\xi_s,s)ds+\lambda g(T,\xi_T)}\biggr].\label{eq 3.2}
 \end{equation}
  Obviously, $J^{\pi^1,\pi^2}(0,x)\geq 1$ for $x\in S$ and $(\pi^1,\pi^2)\in\Pi^1\times\Pi^2$.\\
 Define $J_{*}(t,x):=\sup_{\pi^1\in \Pi_m^1}\inf_{\pi^2\in \Pi_m^2}J^{\pi^1,\pi^2}(t,x)$.
 Note that
 $(\pi^{*1},\pi^{*2})\in\Pi^1\times\Pi^2$ is optimal if and only if
 \begin{align*}
\sup_{\pi^1\in \Pi^1}\inf_{\pi^2\in \Pi^2}J^{\pi^{1},\pi^{2}}(0,x)=\inf_{\pi^2\in \Pi^2}\sup_{\pi^1\in \Pi^1}J^{\pi^{1},\pi^{2}}(0,x)=\inf_{\pi^2\in \Pi^2}J^{\pi^{*1},\pi^{2}}(0,x)=\sup_{\pi^1\in \Pi^1}J^{\pi^{1},\pi^{*2}}(0,x) ~\forall x\in S.
 \end{align*}
 Since the cost rates may be unbounded, the Assumption $3.1(iii)$ is used to guarantee the finiteness of $J^{\pi^1,\pi^2}(t,x)$.
The following Lemma shows the non-exlposion of the state process $\{\xi_t, t\geq 0\}$ and the finiteness of $J^{\pi^1,\pi^2}(t,x)$.
\begin{lemma}\label{lemm 3.1}
Suppose  Assumption \ref{assm 3.1} holds. Then for each pair of strategies $(\pi^1,\pi^2)\in \Pi^1\times\Pi^2,$ the following assertions hold.
	\begin{enumerate}
		\item [(a)] $P^{\pi^1,\pi^2}_x(T_\infty=\infty)=1$, $P^{\pi^1,\pi^2}_x(\xi_t\in S)=1$, and $P^{\pi^1,\pi^2}_x(\xi_0=x)=1$ for each $x\in S$.
		
				\item [(b)] $E^{\pi^1,\pi^2}_x[V(\phi(\xi_t,\hat{t}))]\leq e^{\rho_1 t} [V(\phi(x,t+\hat{t}))+\frac{b_1}{\rho_1}]$,  for each $\hat{t},t\geq 0$ and  $x\in S$.
				\item [(c)] 
				$E^{\pi^1,\pi^2}_{(s,x)}[V(\phi(\xi_t,\hat{t}))]\leq e^{\rho_1 (t-s)} [V(\phi(x,t-s+\hat{t}))+\frac{b_1}{\rho_1}]$,  for each $t\geq s\geq 0$, $\hat{t}\geq 0$ and  $x\in S$ and $(\pi^1,\pi^2)\in \Pi_m^1\times\Pi_m^2$.
		
		\item[(d)]
		\begin{enumerate}
			\item [($d_1$)]  $e^{-L_1 \lambda V(\phi(x,T))}< J^{\pi^1,\pi^2}(0,x)\leq L_1V(\phi(x,T))$ for $x\in S$ and $(\pi^1,\pi^2) \in \Pi^1\times \Pi^2$, where $L_1:=M_2e^{\rho_1 T} [1+\frac{b_1}{\rho_1}]$.
			\item [$(d_2)$]  $e^{-L_2 \lambda V(\phi(x,T-t))}< J^{\pi^1,\pi^2}(t,x)\leq L_2V(\phi(x,T-t))$  for $(t,x)\in [0,T]\times S$ and $(\pi^1,\pi^2)\in \Pi_m^1\times\Pi_m^2$ where $L_2:=M_2e^{\rho_1 (T-t)} [1+\frac{b_1}{\rho_1}]$.				\end{enumerate}
			\end{enumerate}
		\end{lemma} 
		\begin{proof}
For parts (a), (b), and (c) see  \cite{HLG_ORL_2018}, \cite[Proposition 3.1]{HG_ST_2019}.\\
We now prove part (d). For almost surely $\omega:=(x_0,\theta_1,x_1,\cdots,\theta_{k},x_k,\cdots)\in\Omega$, let $k$ (depending on $\omega$) be determined by $T_k(\omega)\leq T<T_{k+1}(\omega)$. Now by equation (\ref{eq 2.3}) and part $(a)$, we say that $k$ is finite, and $\{\xi_t(\omega),t\in [0,T]\}=\{X_0(\omega),\cdots,X_k(\omega)\}$. Now by Assumption \ref{assm 3.1}(iii), and the fact that $\theta_{i+1}=T_{i+1}(\omega)-T_{i}(\omega)$ $(i=0,\cdots,k-1)$, we have
\begin{align*}
&\int_{0}^{T}\int_{B}\int_{A}|c(t,\xi_t,a,b)|\pi^1(da|\omega,t)\pi^2(db|\omega,t)dt\\
&\leq \frac{\sum_{i=0}^{k-1}\theta_{i+1}\log\sqrt{M_2V(\phi(X_i(\omega),T-T_k(\omega)))}+(T-T_k(\omega))\log\sqrt{M_2V(\phi(X_k(\omega),T-T_k(\omega)))}}{(T+1)\lambda}.
\end{align*}
Hence we conclude that 
$\int_{0}^{s}\int_{B}\int_{A}c(t,\xi_t,a,b)\pi^1(da|\omega,t)\pi^2(db|\omega,t)dt$ is real-valued Borel-measurable in $s\in [0,T]$.
 Let $\beta(dt):=I_{[0,T)}(t)dt+\delta_T(dt)$, with $\delta_T(dt)$ being the Dirac measure concentrated on $\{T\}$, and $\tilde{c}(t,x,a,b):=c(t,x,a,b)I_{[0,T)}(dt)+g(T,x)I_{T}(t)$ for each $(t,x,a,b)\in K.$  So for $P^{\pi^1,\pi^2}_x$-almost all $\omega\in \Omega$, by Assumption \ref{assm 3.1}(iii), we can write
\begin{align*}
&\int_{0}^{T}\int_{B}\int_{A}c(t,\xi_t,a,b)\pi^1(da|\omega,t)\pi^2(db|\omega,t)dt+g(T,\xi_T)\\
&=\int_{(0,T]}\int_{B}\int_{A} \tilde{c}(t,\xi_t,a,b)\pi^1(da|\omega,t)\pi^2(db|\omega, t)\beta(dt).
\end{align*}
Now, since $\lambda\leq 1$ and $V(x)\geq 1$,
\interdisplaylinepenalty=0
	\begin{align}
	&E^{\pi^1,\pi^2}_x\biggl[e^{\lambda\int_{0}^{T}\int_{B}\int_{A}c(t,\xi_t,a,b)\pi^1(da|\omega,t)\pi^2(db|\omega,t)dt+\lambda g(T,\xi_T)}\biggr]\nonumber\\
	&=E^{\pi^1,\pi^2}_x\biggl[e^{\lambda\int_{[0,T]}\int_{B}\int_{A}(1+T)\tilde{c}(t,\xi_t,a,b)\pi^1(da|\omega,t)\pi^2(db|\omega,t)\frac{\beta(dt)}{T+1}}\biggr]\nonumber\\
   & \leq 	E^{\pi^1,\pi^2}_x\biggl[\frac{1}{T+1}\int_{[0,T]}e^{\lambda(1+T)\int_{B}\int_{A}|\tilde{c}(t,\xi_t,a,b)|\pi^1(da|\omega,t)\pi^2(db|\omega,t)}\beta(dt)\biggr]\nonumber\\
   &\quad (\text{by the Jensen inequality})\nonumber\\
	&\leq \frac{M_2 }{1+T}E^{\pi^1,\pi^2}_x\biggl[\int_{0}^{T}V(\phi(\xi_t,T-t))dt+V(\phi(\xi_T))\biggr]\nonumber\\
	&\quad (\text{by Assumption \ref{assm 3.1}(iii)})\nonumber\\
	&\leq M_2e^{\rho_1 T} [V(\phi(x,T))+\frac{b_1}{\rho_1}]\nonumber\\
	&\leq M_2e^{\rho_1 T} [1+\frac{b_1}{\rho_1}]V(\phi(x,T))\nonumber\\
	&= L_1 V(\phi(x,T)),~L_1:=M_2e^{\rho_1 T} [1+\frac{b_1}{\rho_1}],\label{eq 3.3}
\end{align}
where the $3$rd inequality is by part (b).
Also, $(T+1)|c(t,x,a,b)|\leq e^{2(T+1)|c(t,x,a,b)|}\leq M_2 V(\phi(x,T-t)) $ and $(T+1)|g(t,x)|\leq e^{2(T+1)|g(t,x)|}\leq M_2 V(\phi(x,T-t)). $ Hence
\begin{align*}
J^{\pi^1,\pi^2}(0,x)&= {E}^{\pi^1,\pi^2}_x\biggl[\exp\biggl({\lambda\int_{0}^{T}\int_{B}\int_{A}c(t,\xi_t,a,b)\pi^1(da|\omega,t)\pi^2(db|\omega,t)dt+\lambda g(T,\xi_T)}\biggr)\biggr]\\
&\geq \exp\biggl({E}^{\pi^1,\pi^2}_x\biggl[{\lambda\int_{0}^{T}\int_{B}\int_{A}c(t,\xi_t,a,b)\pi^1(da|\omega,t)\pi^2(db|\omega,t)dt+\lambda g(T,\xi_T)}\biggr]\biggr)\\
&\geq \exp\biggl(-\frac{1}{T+1}{E}^{\pi^1,\pi^2}_x\biggl[\lambda\int_{0}^{T}\int_{B}\int_{A}M_2V(\phi(\xi_t,T-t))\pi^1(da|\omega,t)\pi^2(db|\omega,t)dt+\lambda M_2V(\xi_T)\biggr]\biggr)\\
&= \exp\biggl(-\frac{M_2\lambda}{(T+1)}{E}^{\pi^1,\pi^2}_x\biggl[\int_{0}^{T}V(\phi(\xi_t,T-t))dt+ V(\xi_T)\biggr]\biggr)\\
&\geq \exp\biggl(-\frac{M_2\lambda}{(T+1)}\biggl[\int_{0}^{T}e^{\rho_1 T} [V(\phi(x,T)+\frac{b_1}{\rho_1}]dt+ e^{\rho_1 T} [V(\phi(x,T))+\frac{b_1}{\rho_1}]\biggr]\biggr)\\
&\geq e^{-L_1 \lambda V(\phi(x,T)}.
\end{align*}
\end{proof}
To study the large enough class of functions and admissible strategies, we need an extension of Feynman-Kac formula. To that end we imposed the following condition. 
\begin{assumption}\label{assm 3.2}
 	There exist a $[1,\infty)$-valued function $V_1$ on $S$, and constants $\kappa>0$, $\rho_2> 0$, $M_3\geq 1$ and $b_2>0$ such that
 	\begin{enumerate}
 		\item[(i)] $q(s,x,a,b)\leq \kappa V(\phi(x,T-s))$, $V^2(x)\leq M_3V_1(x)$, for all $(s,x,a,b)\in K_{[0,T]}$, with the function  $V$ as in Assumption \ref{assm 3.1};
 		\item [(ii)] $\int_{S}V_1^2(\phi(y,t))q(dy |s, x, a,b)\leq \rho_2 V_1^2(\phi(x,t))+b_2$ for each $(t,x, a,b)\in K$, for all $(s,x,a,b)\in K_{[0,T]}$, and $t\in [0,T-s]$.
 	
		\end{enumerate}
 \end{assumption}
We first introduce some frequently used notations. 

Given any real-valued function $W\geq 1$ on $S$ and any Borel set $[0,T]$, a real-valued function $u$ on $[0,T]\times S$ is called $W$-bounded if $\|u\|^\infty_{W}:=\sup_{(t,x)\in [0,T]\times S}\frac{|u(t,x)|}{W(\phi(x,T-t))}< \infty$. Denote $B_{W}([0,T]\times S)$ the Banach space of all $W$-bounded functions. When $W\equiv1$, $B_{1}([0,T]\times S)$ is the space of all bounded functions on $[0,T]\times S.$

Now define, $\mathbb{B}^{ac}_W([0,T]\times S)$ be the collection of all measurable functions $\varphi$ on $[0,T]\times S$ such that $\|\varphi\|_W<\infty$, and for all $(s,x)\in [0,T]\times S$, $\varphi(s+t,\phi(x,t))$ is absolutely continuous in $t\in [0,T-s]$. For a function $\varphi\in \mathbb{B}^{ac}_W([0,T]\times S)$, by \cite[Lemma  2.2]{CDP_SIAM_2016}, there is some measurable function $L^{\phi}\varphi$ on $[0,T]\times S$ satisfying
$$\varphi(s+t,\phi(x,t))-\varphi(s,x)=\int_{0}^{t}L^\phi\varphi(s+v,\phi(x,v))dv~\forall t\in [0,T-s].$$
Here, the function $L^{\phi}\varphi(s+v,\phi(x,v))$ on $[0,T-s]\times S$ coincides with the partial derivative, $\frac{\partial \varphi(s+v,\phi(x,v))}{\partial v}$, of the function $\varphi(s+v,\phi(x,v))$ in $v\in [0,T-s]\times S$ apart from on a null set $Z_\varphi(s,x)\subset [0,T-s]$ with respect to the Lebesgue measure. For such a function $\varphi$, let
$$\mathscr{D}^{\varphi}:=\{(s+t,\phi(x,t))\in [0,T]\times S:t\in Z^c_\varphi(s,x)\}, (s,x)\in [0,T]\times S.$$ Then, the function $L^{\phi}\varphi$ on $[0,T]\times S$ can be defined as below:
	\begin{align}
L^\phi\varphi(s,x):=\left\{ \begin{array}{llll}&\lim_{\Delta s\rightarrow 0}\frac{\varphi(s+\Delta s,\phi(x,\Delta s))-\varphi(s,x)}{\Delta s},~\text{if}~(s,x)\in \mathscr{D}^{\varphi},\\&\text{ arbitrary},~\text{otherwise}.\label{eq 3.4n}
\end{array}\right.
\end{align}
In particular, if $\phi(x,t)\equiv x$ in which case a PDMDP becomes a CTMDP, $L^\phi\varphi(s,x)=\frac{\partial \varphi(s,x)}{\partial s}$, for $(s,x)\in \mathscr{D}^{\varphi}$.
When Assumption \ref{assm 3.2} (ii) is further imposed, Lemma \ref{lemm 3.1}(b) holds with $V$ being replaced by $V^2_1$. Moreover, for $\varphi\in \mathbb{B}^{ac}_W([0,T]\times S)$, we let
$$\|L^\phi\varphi\|^{es}_{V_1}:=\sup_{(s,x)\in \mathscr{D}^{\varphi}}|L^\phi\varphi(s,x)|/{V_1(\phi(x,T-s))},$$ and $$\mathbb{B}^{ac}_{V,V_1}([0,T]\times S):=\{\varphi\in  \mathbb{B}^{ac}_V([0,T]\times S): \|L^\phi\varphi\|^{es}_{V_1}<\infty \}.$$


\theoremstyle{definition}
 \begin{thm}\label{thm 3.1}
 	Suppose Assumptions \ref{assm 3.1} and \ref{assm 3.2} hold.
 \begin{itemize}
 	\item [(a)]  For each $x\in S$, $(\pi^1,\pi^2)\in \Pi^1\times\Pi^2$ and $\varphi\in \mathbb{B}^{ac}_{V,V_1}([0,T]\times S),$ let
 	$$\psi(\omega,t,x):=e^{\int_{0}^{t}\int_A\int_B c(v,\xi_v(\omega),a,b)\pi^1(da|\omega,v)\pi^2(db|\omega,v)dv}\varphi(t,x),~\forall \omega\in \{\omega^{'}\in \Omega|T_{\infty}(\omega^{'})=\infty\}.$$ 
 	\begin{align*}
 	&E^{\pi^1,\pi^2}_x\biggl[\int_{0}^{T}\biggl(L^\phi\psi(\omega,t,\xi_t)+\int_{S}\psi(\omega,t,y)\int_{B}\int_{A}q(dy|t,\xi_t,a,b)\pi^1(da|\omega,t)\pi^2(db|\omega,t)\biggr)dt\biggr]\\
 	&=E^{\pi^1,\pi^2}_x[\psi(\omega,T,\xi_T)]-E^{\pi^1,\pi^2}_x[\psi(\omega,0,x)],
 	\end{align*}
 	where $\psi$ is viwed as a function of the last two arguments when the operator $L^\phi$ is applied and $\{\xi_t,t\geq 0\}$ may be not Markovian since the pair of strategies $(\pi^1,\pi^2)$ may depend on histories.
 	\item [(b)] For each $x\in S$, $(\pi^1,\pi^2)\in \Pi^1_m\times\Pi^2_m$ and $\varphi\in \mathbb{B}^{ac}_{V,V_1}([0,T]\times S),$ let
 	$$\psi(\omega,t,x):=e^{\int_{0}^{t}\int_A\int_B c(v,\xi_v(\omega),a,b)\pi^1(da|\xi_v,v)\pi^2(db|\xi_v,v)dv}\varphi(t,x),~\forall \omega\in \{\omega^{'}\in \Omega|T_{\infty}(\omega^{'})=\infty\}, t\in [s,T].$$ Then, we have 
 	\begin{align*}
 	&E^{\pi^1,\pi^2}_{(s,x)}\biggl[\int_{s}^{T}\biggl(L^\phi\psi(\omega,t,\xi_t)+\int_{S}\psi(\omega,t,y)\int_{B}\int_{A}q(dy|t,\xi_t,a,b)\pi^1(da|\xi_t,t)\pi^2(db|\xi_t,t)\biggr)dt\biggr]\\
 	&=E^{\pi^1,\pi^2}_{(s,x)}[\psi(\omega,T,\xi_T)]-E^{\pi^1,\pi^2}_{(s,x)}[\psi(\omega,0,\xi_s)].
 	\end{align*}
\end{itemize}
\end{thm}
\begin{proof}
	For proof, see \cite[Theorem 3.1]{HLG_ORL_2018}.
	\end{proof}
The next statement is a verification theorem, which provides that the solution of the risk-sensitive finite-horizon optimality equation (Shapley equation) has a stochastic representation.
\begin{corollary}\label{corr 3.1}
 Suppose Assumptions \ref{assm 3.1} and \ref{assm 3.2} are satisfied.
  If there exist a function $\varphi\in \mathbb{B}^{ac}_{V,V_1}([0,T]\times S),$ such that
 \interdisplaylinepenalty=0
 \begin{align}
 &L^{\phi}\varphi(t,x)+\biggl\{\lambda c(t,x,\mu,\nu)\varphi(t,x)\nonumber\\
 &+\int_{S}\varphi(t,y)q(dy|t,x,\mu,\nu)\biggr\}\geq (\leq )0,~\forall (\mu,\nu)\in P(A(t,x))\times P(B(t,x)), (t,x)\in \mathscr{D}^{\varphi},\nonumber\\
 &\text{ where}~\varphi(T,x)=e^{g(T,x)},~~~x\in S,\label{eq 3.13}
 \end{align}
 then 
 \begin{enumerate}
 	\item [(a)]

\interdisplaylinepenalty=0
\begin{align}
\varphi(0,x) \leq(\geq ) J^{\pi^1,\pi^2}(0,x),\forall (\pi^1,\pi^2)\in \Pi^1\times \Pi^2, \text{and}~x\in S.\label{eq 3.14}
\end{align}
\item  [(b)] 
\interdisplaylinepenalty=0
\begin{align}
\varphi(t,x) \leq(\geq ) J^{\pi^1,\pi^2}(t,x), \forall (\pi^1,\pi^2)\in \Pi_m^1 \times  \Pi_m^2,  \text{and}~(t,x)\in [0,T]\times S. \label{eq 3.15}
\end{align}
\end{enumerate}
\begin{proof}
	\begin{enumerate}
		\item [(a)] 	For almost all $\omega\in \Omega$ with respect to ${P}^{\pi^1,\pi^2}_x$ and almost all $t\in [0,T]$, $(t,\xi_t)\in \mathscr{D}^{\varphi}$, using equation (\ref{eq 3.13}), we have 
		\interdisplaylinepenalty=0
		
	 \begin{align}
	&L^{\phi}\varphi(t,x)+\biggl\{\lambda c(t,x,\mu,\nu)\varphi(t,x)\nonumber\\
	&+\int_{S}\varphi(t,y)q(dy|t,x,\mu,\nu)\biggr\}\geq 0,~\forall (\mu,\nu)\in P(A(t,x))\times P(B(t,x)), (t,x)\in \mathscr{D}^{\varphi},\nonumber\\
	&\text{ where}~\varphi(T,x)=e^{g(T,x)},~~~x\in S.\label{eq 3.16}
	\end{align}

		Let $(\pi^1,\pi^2)\in \Pi^1\times \Pi^2$, be arbitrarily fixed pair of strategies. Now by Theorem \ref{thm 3.1}, we have the Feynman-Kac formula, corresponding the pair of strategies $(\pi^{1},\pi^2)$, as
		\interdisplaylinepenalty=0
		\begin{align}
		&J^{\pi^{1},\pi^2}(0,x)-\varphi(0,x)\nonumber\\
		&=E^{\pi^{1},\pi^2}_x\biggl[e^{\lambda \int_{0}^{T}\int_{B}\int_{A}c(v,\xi_v,a,b)\pi^{1}(da|\omega,v)\pi^2(db|\omega,v)dv}\varphi(T,\xi_T)\biggr]-\varphi(0,x)\nonumber\\
		&=E^{\pi^{1},\pi^2}_x\biggl[\int_{0}^{T}e^{\lambda \int_{0}^{t}\int_{B}\int_{A}c(v,\xi_v,a,b)\pi^{1}(da|\omega,v)\pi^2(db|\omega,v)dv}\int_{B}\int_{A}\biggl(\lambda c(t,\xi_t,a,b)\varphi(t,\xi_t)+L^\phi\varphi(t,\xi_t)\nonumber\\
		&\quad+\int_{S}\varphi(t,y)q(dy|t,\xi_t,a,b)\biggr)\pi^{1}(da|\omega,t)\pi^2(db|\omega,t)dt\biggr]\label{eq 3.17},
		\end{align}
		here the 1st equality holds because $\varphi(T,x)=e^{\lambda g(T,x)}$, see (\ref{eq 3.13}).\\
		From (\ref{eq 3.16}) and (\ref{eq 3.17}), we have
		$J^{\pi^{1},\pi^2}(0,x)\geq \varphi(0,x)$.
		\item [(b)] The proof of this part is similar as part (a).	
	\end{enumerate}

\end{proof}

\end{corollary}

\begin{thm}\label{Theorem 3.2}
	Grant Assumptions \ref{assm 3.1} and \ref{assm 3.2}.
	For every pair of  Markov strategies $(\pi^{1},\pi^{2})\in \Pi_{m}^1\times\Pi_{m}^2$, $J^{\pi^1,\pi^{2}}(\cdot,\cdot)$ is a solution to the following differential equation
	\interdisplaylinepenalty=0
	\begin{align}
\left\{ \begin{array}{llll}&L^{\phi}\varphi(t,x)+ c(t,x,\pi^{1}(\cdot|x,t),\pi^{2}(\cdot|x,t))\varphi(t,x)+\int_{S}\varphi(t,y)q(dy|t,x,\pi^{1}(\cdot|x,t),\pi^{2}(\cdot|x,t))=0,\\&\varphi(T,x)=e^{g(T,x)}.\label{eq 3.13n}
\end{array}\right.
\end{align}
	\end{thm}
	\begin{proof}
		In view of Lemma \ref{lemm 3.1}, we have $\|J^{\pi^1,\pi^{2}}\|_{V}\leq M_2e^{\rho_1 T} [1+\frac{b_1}{\rho_1}]<\infty$. Now, fix arbitrarily $(s,x)\in [0,T]\times S$. Then by conditioning on the
		first-jump time and the post-jump state, using the jump property of the process $\{\xi_t\}$ yields that
		\begin{align}
	J^{\pi^1,\pi^{2}}(s,x)&=e^{-\int_{s}^{T}q(v,\phi(x,v-s),\pi^{1}(\cdot|\phi(x,v-s),v),\pi^{2}(\cdot|\phi(x,v-s,v)))dv}\nonumber\\
	&\times e^{\int_{s}^{T}r(v,\phi(x,v-s),\pi^{1}(\cdot|\phi(x,v-s),v),\pi^{2}(\cdot|\phi(x,v-s),v))dv}e^{g(T,\phi(x,T-s))}\nonumber\\
	&+\int_{s}^{T}\int_{S\backslash \{\phi(x,z-s)\}}	J^{\pi^1,\pi^{2}}(z,y)q(dy|z,\phi(x,z-s),\pi^{1}(\cdot|\phi(x,z-s),z),\pi^{2}(\cdot|\phi(x,z-s),z))\nonumber\\
	&\cdot e^{-\int_{s}^{z}q(v,\phi(x,v-s),\pi^{1}(\cdot|\phi(x,v-s),v),\pi^{2}(\cdot|\phi(x,v-s),v))dv}\nonumber\\	&\times e^{\int_{s}^{z}r(v,\phi(x,v-s),\pi^{1}(\cdot|\phi(x,v-s),v),\pi^{2}(\cdot|\phi(x,v-s),v))dv}dz,\label{eq 3.21n}
		\end{align}
		
		for every $(s,x)\in [0,T]\times S$. It is obvious that $	J^{\pi^1,\pi^{2}}(T,x)=e^{g(T,x)}$. For every $s\in [0,T)$, replacing respectively $s$ with $s+t$, and $x$ with $\phi(x,t)~(t\in [0,T-s])$ in (\ref{eq 3.21n}), and then multiplying by
		\begin{align*}
	&e^{-\int_{s}^{s+t}q(v,\phi(x,v-s),\pi^{1}(\cdot|\phi(x,v-s),v),\pi^{2}(\cdot|\phi(x,v-s),v))dv}\nonumber\\
	&\times e^{\int_{s}^{s+t}r(v,\phi(x,v-s),\pi^{1}(\cdot|\phi(x,v-s),v),\pi^{2}(\cdot|\phi(x,v-s),v))dv}
		\end{align*} 
		both sides of (\ref{eq 3.21n}), we obtain
			\begin{align*}
		&e^{-\int_{s}^{s+t}q(v,\phi(x,v-s),\pi^{1}(\cdot|\phi(x,v-s),v),\pi^{2}(\cdot|\phi(x,v-s),v))dv}\nonumber\\
		&\times e^{\int_{s}^{s+t}r(v,\phi(x,v-s),\pi^{1}(\cdot|\phi(x,v-s),v),\pi^{2}(\cdot|\phi(x,v-s),v))dv}\times J^{\pi^1,\pi^{2}}(s+t,\phi(x,t))\nonumber\\
			&=e^{-\int_{s}^{T}q(v,\phi(x,v-s),\pi^{1}(\cdot|\phi(x,v-s),v),\pi^{2}(\cdot|\phi(x,v-s),v))dv}\nonumber\\
			&\times e^{\int_{s}^{T}r(v,\phi(x,v-s),\pi^{1}(\cdot|\phi(x,v-s),v),\pi^{2}(\cdot|\phi(x,v-s),v))dv}e^{g(T,\phi(x,T-s))}\nonumber\\
			&+\int_{s+t}^{T}\int_{S\backslash \{\phi(x,z-s)\}}J^{\pi^1,\pi^{2}}(z,y)q(dy|z,\phi(x,z-s),\pi^{1}(\cdot|\phi(x,z-s),z),\pi^{2}(\cdot|\phi(x,z-s),z))\nonumber\\
			&\cdot e^{-\int_{s}^{z}q(v,\phi(x,v-s),\pi^{1}(\cdot|\phi(x,v-s),v),\pi^{2}(\cdot|\phi(x,v-s),v))dv}e^{\int_{s}^{z}r(v,\phi(x,v-s),\pi^{1}(\cdot|v,\phi(x,v-s)),\pi^{2}(\cdot|v,\phi(x,v-s)))dv}dz,
		\end{align*} 
		which shows that $J^{\pi^1,\pi^{2}}(s+t,\phi(x,t))$ is absolutely continuous in $t\in [0,T-s]$, and thus, is differentiable almost everywhere on $[0,T-s]$. Hence, differentiating both sides of the resulting equality with respect to $t$, and then dividing by 
		\begin{align*}
		&e^{-\int_{s}^{s+t}q(v,\phi(x,v-s),\pi^{1}(\cdot|\phi(x,v-s),v),\pi^{2}(\cdot|\phi(x,v-s),v))dv}\nonumber\\
		&\times e^{\int_{s}^{s+t}r(v,\phi(x,v-s),\pi^{1}(\cdot|\phi(x,v-s),v),\pi^{2}(\cdot|\phi(x,v-s),v))dv}
		\end{align*}
		both sides of the resulting equality yield
		\begin{align*}
		&L^{\phi}J^{\pi^1,\pi^{2}}(s+t,\phi(x,t))+r(s+t,\phi (x,t),\pi^{1}(\cdot|\phi(x,t),s+t),\pi^{2}(\cdot|\phi(x,t),s+t))J^{\pi^{1},\pi^{2}}(s+t,\phi(x,t))\\
		&+\int_SJ^{\pi^{1},\pi^{2}}(s+t,y)q(dy|s+t,\phi(x,t),\pi^{1}(\cdot|\phi(x,t),s+t),\pi^{2}(\cdot|\phi(x,t),s+t))=0~\forall t\in Z^c_{\varphi}(s,x).
		\end{align*}
		This implies that
			\begin{align*}
		L^{\phi}J^{\pi^1,\pi^{2}}(s,x)&+r(s,x,\pi^{1}(\cdot|x,s),\pi^{2}(\cdot|x,s))J^{\pi^{1},\pi^{2}}(s,x)\\
		&+\int_SJ^{\pi^{1},\pi^{2}}(s,y)q(dy|s,x,\pi^{1}(\cdot|x,s),\pi^{2}(\cdot|x,s))=0~\forall (s,x)\in \mathscr{D}^{J^{\pi^1,\pi^{2}}}.
		\end{align*}
		Now,
		by (\ref{eq 3.13n}) and Assumption \ref{assm 3.1} and \ref{assm 3.2}, we obtain
		\begin{align*}
|	L^{\phi}J^{\pi^1,\pi^{2}}(s,x)|&\leq [M_2M_3\|J^{\pi^1,\pi^{2}}\|_{V}+\|J^{\pi^1,\pi^{2}}\|_{V}(\rho_1M_3+b_1+2\kappa M_3)]\\
&\times V_1(\phi(x,T-s)),~(s,x)\in \mathscr{D}^{J^{\pi^1,\pi^{2}}}.
		\end{align*}
	Hence, $\|L^{\phi}{J^{\pi^1,\pi^{2}}}\|^{es}_{V_1}<\infty$. Consequently, $J^{\pi^1,\pi^{2}}\in \mathbb{B}^{ac}_{V,V_1}([0,T]\times S)$.
	Now, if $\varphi(s,x)\in \mathbb{B}^{ac}_{V,V_1}([0,T]\times S)$ is also a solution to eq. (\ref{eq 3.13n}), by Corrolary \ref{corr 3.1}, we must have $J^{\pi^1,\pi^{2}}(s,x)=\varphi(s,x)$, for all $(s,x)\in [0,T]\times S$. In other word, the solution of equation (\ref{eq 3.13n}) is unique.
\end{proof}
\section{The existence of saddle-point equilibria}
In this section, we prove that the equation (\ref{eq 3.21}) is the optimality equation  for the cost criterion (\ref{eq 3.1}) and the equation (\ref{eq 3.21}) has a solution in $ \mathbb{B}^{ac}_{V,V_1}([0,T]\times S$.  Also we prove the existence of saddle-point equilibria by using the optimality equation (\ref{eq 3.21}). We impose the following continuty and compactness Assumptions, which guarantee the existence of a pair of optimal strategies.
 \begin{assumption}\label{assm 4.1}
 	\begin{enumerate}
 		\item [(i)] For each $(t,x)\in [0,T]\times S$, $A(t,x)$ and $B(t,x)$ are compact.
 		\item [(ii)] For each $t\in [0,T]$, $x\in S$, the function  $q(\cdot |t, x, a,b)$ is  continuous in $(a,b)\in A(t,x)\times B(t,x)$.
 			\item [(iii)] For each $(t,x)\in [0,T]\times S,$ the function $c(t,x,a,b)$ is continuous in $(a,b)\in A(t,x)\times B(t,x)$, and the function $\displaystyle \int_{S}V(\phi(y,T-t))q(dy|t,x,a,b)$ is continuous in $(a,b)\in A(t,x)\times B(t,x)$, with $V$ as in Assumption \ref{assm 3.1}.
 		\end{enumerate}
 \end{assumption}
 By closely mimicking the arguments in \cite{GH4}, or \cite[Lemma 8.3.7]{HL2}, one can easily get the following result, which is essential to prove the  existence of the optimal saddle point equilibrium; we omit the details. 
 \begin{lemma}\label{lemm 4.1}
 	Under Assumptions \ref{assm 3.1}, \ref{assm 3.2} and \ref{assm 4.1}, the functions 
 	$$c(t,x,\mu,\nu) \; \; \mbox{and} \; \; \int_{S}q(dy|t,x,\mu,\nu)u(t,y)$$ are continuous  at $(\mu,\nu)$ on $P(A(t,x))\times P(B(t,x))$ for each fixed $u \in \mathbb{B}^{ac}_{V,V_1}([0,T]\times S)$ and $x\in S$.
 \end{lemma}
 In the next Proposition, we prove that the optimality equation \ref{eq 3.1}, has a solution if the transition and cost rates are bounded. 
\begin{proposition}\label{prop 4.1}
	Suppose that the transition and cost rates are bounded, i,e.,
	\begin{align*}
	\sup_{x\in S}q^{*}(x)<\infty,\ \ \sup_{(t,x,a,b)\in K}|c(t,x,a,b)|< \infty,\ \ \sup_{(t,x)\in [0,T]\times S}|g(t,x)|< \infty.
	\end{align*}
	Also for each $x\in S$, $t\in [0,T]$, $A(t,x)$, $B(t,x)$ are compact, $c(t,x,a,b)$ is continuous in $(a,b)\in A(t,x)\times B(t,x)$ and for each bounded measurable function $u$ on $S$, the function $\int_{S}u(\phi(y,T-t))q(dy|t,x,a,b)$ is continuous in $(a,b)\in A(t,x)\times B(t,x)$, 
	then there exists a unique $\varphi\in \mathbb{B}^{ac}_{1,1}([0,T]\times S),$ and some $(\pi^{*1},\pi^{*2})\in \Pi_{m}^1\times \Pi_{m}^2$  satisfying 	the following equations (i.e., (\ref{eq 3.21})-(\ref{eq 3.22a}))  \begin{align}
	&-L^{\phi}\varphi(t,x)\nonumber\\
	&=\sup_{\mu\in P(A(t,x))}\inf_{\nu\in P(B(t,x))}\biggl\{\lambda c(t,x,\mu,\nu)\varphi(t,x)+\int_{S}\varphi(t,y)q(dy|t,x,\mu,\nu)\biggr\}\nonumber\\
	&=\inf_{\nu\in P(B(t,x))}\sup_{\mu\in P(A(t,x))}\biggl\{\lambda c(t,x,\mu,\nu)\varphi(t,x)+\int_{S}\varphi(t,y)q(dy|t,x,\mu,\nu)\biggr\}\nonumber\\
	&=\inf_{\nu\in P(B(t,x))}\biggl\{\lambda c(t,x,\pi^{*1}(\cdot|x,t),\nu)\varphi(t,x)+\int_{S}\varphi(t,y)q(dy|t,x,\pi^{*1}(\cdot|x,t),\nu)\biggr\}\nonumber\\
	&=\sup_{\mu\in P(A(t,x))}\biggl\{\lambda c(t,x,\mu,\pi^{*2}(\cdot|x,t))\varphi(t,x)+\int_{S}\varphi(t,y)q(dy|t,x,\mu,\pi^{*2}(\cdot|x,t))\biggr\}\nonumber\\
		&=\lambda c(t,x,\pi^{*1}(\cdot|x,t),\pi^{*2}(\cdot|x,t))\varphi(t,x)+\int_{S}\varphi(t,y)q(dy|t,x,\pi^{*1}(\cdot|x,t),\pi^{*2}(\cdot|x,t)),~\forall (t,x)\in \mathscr{D}^{\varphi}\nonumber\\
	&\text{ where}~\varphi(T,x)=e^{g(T,x)},~~~x\in S.\label{eq 3.21}
	\end{align}
	
	 	Also,
	 		\begin{enumerate}
		\item [(a)]

		\interdisplaylinepenalty=0
		\begin{align}
		\varphi(0,x) &=\sup_{\pi^1\in \Pi^1}\inf_{\pi^2\in \Pi^2} J^{\pi^1,\pi^2}(0,x)
		=\inf_{\pi^2\in \Pi^2}\sup_{\pi^1\in \Pi^1}J^{\pi^1,\pi^2}(0,x)\nonumber\\
		&= \inf_{\pi^2\in \Pi^2}J^{\pi^{*1},\pi^2}(0,x)
		=\sup_{\pi^1\in \Pi^1}J^{\pi^1,\pi^{*2}}(0,x).\label{eq 3.22}
		\end{align}
		\item  [(b)] 
		\interdisplaylinepenalty=0
		\begin{align}
		\varphi(t,x)& =\sup_{\pi^1\in \Pi_m^1}\inf_{\pi^2\in \Pi_m^2} J^{\pi^1,\pi^2}(t,x)
		=\inf_{\pi^2\in \Pi_m^2}\sup_{\pi^1\in \Pi_m^1}J^{\pi^1,\pi^2}(t,x)\nonumber\\
		&= \inf_{\pi^2\in \Pi_m^2}J^{\pi^{*1},\pi^2}(t,x)
		=\sup_{\pi^1\in \Pi_m^1}J^{\pi^1,\pi^{*2}}(t,x). \label{eq 3.22a}
		\end{align}
	\end{enumerate}
	\begin{proof}
	Since $q^{*}(x)=\sup_{t\geq 0,a\in A(t,x),b\in B(t,x)}q(t,x,a,b)$ is bounded, we may use the Lyapunov function $V_0(\cdot)\equiv 1$ such that	$\int_{S}q(dy|t,x,a,b)V_0(\phi(y,T-t))\leq \rho_1 V_0(x)+b_1$  for $x\in S$. Also, we have $\|q\|:=\sup_{x\in S}q^{*}(x)<\infty$, $\|c\|:=\sup_{(t,x,a,b)\in K}c(t,x,a,b)<\infty$ and $\|g\|:=\sup_{(t,x)\in [0,T]\times S}g(t,x)<\infty$.
		Now let us define an nonlinear operator $T$ on $B_{1}([0,1]\times S)$ as follows:
		\begin{align*}
		\Gamma u(s,x)=e^{\lambda g(T,\phi(x,T-s))}&+\int_{s}^{T}\sup_{\mu\in P(A(z,\phi(x,z-s)))}\inf_{\nu\in P(B(z,\phi(x,z-s)))}\biggl[\int_{S}q(dy|z,\phi(x,z-s),\mu,\nu)u(z,y)\\
		&+c(z,\phi(x,z-s),\mu,\nu)u(z,\phi(x,z-s))\biggr]dz,
		\end{align*}
		where $u\in B_{1}([0,T]\times S)$ and $(s,x)\in [0,T]\times S$. Then by Fan's minimax theorem, see \cite[Theorem 3]{Fan} we have
		\interdisplaylinepenalty=0
		\begin{align*}
	\Gamma u(s,x)=e^{\lambda g(T,\phi(x,T-s))}&+\int_{s}^{T}\inf_{\nu\in P(B(z,\phi(x,z-s)))}\sup_{\mu\in P(A(z,\phi(x,z-s)))}\biggl[\int_{S}q(dy|z,\phi(x,z-s),\mu,\nu)u(z,y)\\
	&+c(z,\phi(x,z-s),\mu,\nu)u(z,\phi(x,z-s))\biggr]dz,
	\end{align*}
		By using the Assumption \ref{assm 3.1} and the facts that $c$ and $q$ are bounded, we obtain
		\interdisplaylinepenalty=0
		\begin{align*}
		&\sup_{s\in [0,T]}\sup_{x\in S}|	\Gamma u(s,x)|\\
		&\leq e^{\|g\|}+\int_{s}^{T}\inf_{\nu\in P(B(z,\phi(x,z-s)))}\sup_{\mu\in P(A(z,\phi(x,z-s)))}\biggl\{\sup_{x\in S}\biggl[\int_{S}|q(dy|z,\phi(x,z-s),\mu,\nu)|u(z,y)\biggr]\\
		&+\|c\|\sup_{x\in S}u(z,\phi(x,z-s))\biggr\}dz\\
		&\leq e^{\|g\|}+{\|u\|_1^\infty}\biggl\{\int_{s}^{T}\sup_{\mu\in P(B(z,\phi(x,z-s)))}\sup_{\nu\in P(B(z,\phi(x,z-s)))}\sup_{x\in S}\biggl(2q(s,x,\mu,\nu)\biggr)ds+\|c\|(T-s)\biggr\}\\
		&\leq e^{\|g\|}+{\|u\|_1^\infty}(2\|q\|+\|c\|)(T-s).
		\end{align*}
		Therefore, $	\Gamma $ is a nonlinear operator from $B_{1}([0,T]\times S)$ to $B_{1}([0,T]\times S)$. Next, we prove that $	\Gamma $ is a $m$-step contraction operator. For any $g_1,g_2\in B_{1}([0,T]\times S)$, we have
		\interdisplaylinepenalty=0
		\begin{align}
		\sup_{x\in S}|	\Gamma g_1(t,x)-	\Gamma g_2(t,x)|
		&\leq \int_{t}^{T}\biggl(2\|q\|+\|c\|\biggr)\|g_1-g_2\|ds\nonumber\\
		&= \biggl(2\|q\|+\|c\|\biggr)\|g_1-g_2\|^\infty_1(T-t).\label{eq 4.1}
		\end{align}
		Furthermore, we can prove the following estimation by induction:
		\begin{equation}
		\sup_{x\in S}|	\Gamma ^lg_1(t,x)-	\Gamma ^lg_2(t,x)|\leq \frac{\|g_1-g_2\|^\infty_1}{ l!}\biggl[\biggl(2\|q\|+\|c\|\biggr)(T-t)\biggr]^l~~\forall~l\geq1.\label{eq 4.2}
		\end{equation}
		By (\ref{eq 4.2}) and the inductive assumption we have
		\interdisplaylinepenalty=0
		\begin{align*}
		&\sup_{x\in S}|	\Gamma ^{l+1}g_1(t,x)-	\Gamma ^{l+1}g_2(t,x)|\\
		&\leq \biggl(2\|q\|+\|c\|\biggr)\int_{t}^{T}\sup_{x\in S}|	\Gamma ^lg_1(s,x)-	\Gamma ^lg_2(s,x)|ds\\
		&	\leq  \frac{\|g_1-g_2\|^\infty_1}{ l!}\biggl[\biggl(2\|q\|+\|c\|\biggr)\biggr]^{l+1}\int_{t}^{T}(T-s)^l ds\\
		&=\frac{\biggl((2\|q\|+\|c\|)(T-t)\biggr)^{l+1}}{l+1!}\|g_1-g_2\|^\infty_1.
		\end{align*}
		Since $\sum_{k\geq 1}\frac{\biggl((2\|q\|+\|c\|)(T-t)\biggr)^{k}}{k!}<\infty$, there exists $m$ such that 
		$\beta:=\frac{\biggl((2\|q\|+\|c\|)(T-t)\biggr)^{m}}{m!}<1,$
		which implies that $\|	\Gamma ^m g_1-	\Gamma ^m g_2\|_1^\infty\leq 
		\beta \|g_1-g_2\|^\infty_1$. Therefore, $\Gamma $ is a $m$-step contraction operator on $B_{1}([0,T]\times S)$. So, by Banach fixed point theorem, there exists a unique bounded function $\varphi\in B_{1}([0,T]\times S)$ such that $	\Gamma \varphi(t,x)=\varphi(t,x)$; that is,
		\interdisplaylinepenalty=0
		\begin{align*}
		-L^\phi\varphi(t,x)=&\inf_{\nu\in P(B(t,x))}\sup_{\mu\in P(A(t,x))}\biggl[\int_{S}q(dy|t,x,\mu,\nu)	\varphi(t,y)+c(t,x,\mu,\nu)	\varphi(t,x)\biggr]\\
		=&\sup_{\mu\in P(A(t,x))}\inf_{\nu\in P(B(t,x))}\biggl[\int_{S}q(dy|t,x,\mu,\nu)	\varphi(t,y)+c(t,x,\mu,\nu)	\varphi(t,x)\biggr].
		\end{align*}
	Note that $\varphi(T,x)=e^{\lambda g(T,x)}$. By the above equation, Assumptions \ref{assm 3.1} and \ref{assm 3.2}, we say that $\varphi(t,x)\in \mathbb{B}_{1,1}([0,T]\times S)$. Now by measurable selection theorem in \cite{N}, we have there exists a pair of strategies $(\pi^{*1},\pi^{*2})\in \Pi_m^1\times\Pi_m^2$ satisfying equation (\ref{eq 3.21}).
	Next, from (\ref{eq 3.21}), we obtain
		 \begin{align}
&L^{\phi}\varphi(t,x)+\inf_{\nu\in P(B(t,x))}\biggl\{\lambda c(t,x,\pi^{*1}(\cdot|x,t),\nu)\varphi(t,x)+\int_{S}\varphi(t,y)q(dy|t,x,\pi^{*1}(\cdot|x,t),\nu)\biggr\}=0~\forall (t,x)\in \mathscr{D}^{\varphi}\nonumber\\
	&\text{ where}~\varphi(T,x)=e^{g(T,x)},~~~x\in S.\label{eq 3.14New}
	\end{align}
	Hence, 
	\begin{align}
	&L^{\phi}\varphi(t,x)+\biggl\{\lambda c(t,x,\pi^{*1}(\cdot|x,t),\nu)\varphi(t,x)+\int_{S}\varphi(t,y)q(dy|t,x,\pi^{*1}(\cdot|x,t),\nu)\biggr\}\geq 0~\forall (t,x)\in \mathscr{D}^{\varphi}\nonumber\\
	&\text{ where}~\varphi(T,x)=e^{g(T,x)},~~~x\in S.\label{eq 3.15New}
	\end{align}
	Then by Corollary \ref{corr 3.1}, we obtain
	$\varphi(0,x) \leq J^{\pi^{*1},\pi^2}(0,x), \forall \pi^2\in \Pi^2$.	
	Since $\pi^2\in \Pi^2$ is arbitrary strategy for player 2, we have
	$$\inf_{\pi^2\in \Pi^2} J^{\pi^{*1},\pi^2}(0,x)\geq \varphi(0,x).$$
	So,
	\begin{align}
	\sup_{\pi^1\in \Pi^1}\inf_{\pi^2\in \Pi^2} J^{\pi^1,\pi^2}(0,x)\geq \inf_{\pi^2\in \Pi^2}J^{\pi^{*1},\pi^2}(0,x)\geq \varphi(0,x).\label{eq 3.18}
	\end{align}
	Similarly, we have
		\begin{align}
\inf_{\pi^2\in \Pi^2} 	\sup_{\pi^1\in \Pi^1}J^{\pi^1,\pi^2}(0,x)\leq  \varphi(0,x).\label{eq 3.18a}
	\end{align}
	Also, from (\ref{eq 3.21}), we get 
		 \begin{align}
& c(t,x,\pi^{*1}(\cdot|x,t),\pi^{*2}(\cdot|x,t))\varphi(t,x)+\int_{S}\varphi(t,y)q(dy|t,x,\pi^{*1}(\cdot|x,t),\pi^{*2}(\cdot|x,t)),~\forall (t,x)\in \mathscr{D}^{\varphi}\nonumber\\
	&\text{ where}~\varphi(T,x)=e^{g(T,x)},~~~x\in S.\label{eq 3.19}
	\end{align}
	Using (\ref{eq 3.19}) and Theorem \ref{Theorem 3.2}, we conclude that
		\begin{align}
 J^{\pi^{*1},\pi^{*2}}(0,x)= \varphi(0,x).\label{eq 3.20}
	\end{align}
	Combining (\ref{eq 3.18}), (\ref{eq 3.18a}) and (\ref{eq 3.20}), we have
\interdisplaylinepenalty=0
\begin{align*}
\varphi(0,x)& =\sup_{\pi^1\in \Pi^1}\inf_{\pi^2\in \Pi^2} J^{\pi^1,\pi^2}(0,x)
=\inf_{\pi^2\in \Pi^2}\sup_{\pi^1\in \Pi^1}J^{\pi^1,\pi^2}(0,x)\nonumber\\
&= J^{\pi^{*1},\pi^{*2}}(0,x). 
\end{align*}
Similarly, we obtain
\interdisplaylinepenalty=0
\begin{align*}
\varphi(t,x)& =\sup_{\pi^1\in \Pi_m^1}\inf_{\pi^2\in \Pi_m^2} J^{\pi^1,\pi^2}(t,x)
=\inf_{\pi^2\in \Pi_m^2}\sup_{\pi^1\in \Pi_m^1}J^{\pi^1,\pi^2}(t,x)\nonumber\\
&= \inf_{\pi^2\in \Pi_m^2}J^{\pi^{*1},\pi^2}(t,x)
=\sup_{\pi^1\in \Pi_m^1}J^{\pi^1,\pi^{*2}}(t,x)\nonumber\\
&= J^{\pi^{*1},\pi^{*2}}(t,x).
\end{align*}
			\end{proof}
\end{proposition} 

\begin{proposition}\label{prop 4.2}
	Suppose Assumptions \ref{assm 3.1}, \ref{assm 3.2} and \ref{assm 4.1} hold. Then if $\|q\|<\infty$, $\|c\|<\infty$, $c(t,x,a,b)\geq 0$ and $g(t,x)\geq 0$, for all $(t,x,a,b)\in K$, the followings are true.
	\begin{enumerate}
		\item [(a)] There exists a unique solution $\varphi(t,x)$ and a pair of strategies $(\pi^{*1},\pi^{*2})\in \Pi_m^1\times\Pi_m^2$ satisfying (\ref{eq 3.21}), (\ref{eq 3.22}) and (\ref{eq 3.22a}).
		\item [(b)] $J_{*}(t,x)$ (and so $\varphi(t,x)$) is decreasing in $t\in [0,T]$ for fixed $x\in S$.
		\begin{proof}
			\begin{enumerate}
				\item [(a)] The proof of this part follows from Proposition \ref{prop 4.1}.
				\item [(b)]
		Fix any $s,t\in [0,T]$ with $s<t$. Also fix any $(\pi^1,\pi^2)\in \Pi_m^1\times\Pi_m^2$.
			Now for each $x\in S$, define a Markov strategy corresponding to $\pi^1\in \Pi_m^1$ as
			\begin{align}
	\pi^1_{s,t}	(db|x,v)=\left\{ \begin{array}{ll}&\pi^1(db|x,v+t-s)~\text{if}~v\geq s\\
	&\pi^1(db|x,v)~\text{otherwise}.
		\end{array}\right.\label{eq 4.3}
		\end{align}	
		Similarly, corresponding to the strategy $\pi^2\in \Pi_m^2$, we define $\pi^2_{s,t}	$. 
		Then, for each $(v,x)\in [s,s+T-t]\times S$,
		$q(dy|t,x,\pi^1_{s,t}	(db|x,v),	\pi^2_{s,t}	(db|x,v))=q(dy|t,x,\pi^1(db|x,v+t-s),\pi^2(db|x,v+t-s)),~c(t,x,\pi^1_{s,t}(db|x,v),	\pi^2_{s,t}	(db|x,v))=c(t,x,\pi^1(db|x,v+t-s),\pi^2(db|x,v+t-s)).$\
		Now define 
		\begin{align}
	&	J^{\pi^1,\pi^2}(s\sim t,x):= {E}^{\pi^1,\pi^2}_{(s,x)}\biggl[e^{\lambda\int_{s}^{t}c(t,\xi_t,\pi^1(da|\xi_v,v),\pi^2(db|\xi_v,v))dv+\lambda g(T,\xi_T)}\biggr].\label{eq 4.4}
		\end{align}
			\begin{align}
	J_{*}(s\sim t,x):=\sup_{\pi^1\in\Pi_m^1}\inf_{\pi^2\in \Pi_m^2}J^{\pi^1,\pi^2}(s\sim t,x).\label{eq 4.4a}
		\end{align}
		Now by the Markov property of $\{\xi_t,t\geq 0\}$ under any pair of Markov strategies $(\pi^1,\pi^2)$ and (\ref{eq 4.3})-(\ref{eq 4.4}), we have $\xi_u$ under strategies $\pi^1,\pi^2$ and $\xi_t=x$ has the same distribution with $\xi_{u+s-t}$ under the strategies $\pi^1_{s,t}$, $\pi^2_{s,t}$ and $\xi_s=x$ for any $t\leq u \leq T$.
		Consequently, we have
		$J^{\pi^1,\pi^2}(t\sim T,x)=J^{\pi^1_{s,t},\pi^2_{s,t}}(s\sim T+s-t,x)$.
		Next, it is easy to note that
		\begin{align*}
		\inf_{\pi^2\in \Pi_m^2}J^{\pi^1,\pi^2}(t\sim T,x)&=\inf_{\pi^2\in \Pi_m^2}J^{\pi^1_{s,t},\pi^2_{s,t}}(s\sim T+s-t,x)\\
		&\geq \inf_{\pi^2_{s,t}\in \Pi_m^2}J^{\pi^1_{s,t},\pi^2_{s,t}}(s\sim T+s-t,x)~\forall\pi^1\in \Pi^1_m.
		\end{align*}
		Similarly,
			\begin{align*}
	\inf_{\pi^2_{s,t}\in \Pi_m^2}J^{\pi^1_{s,t},\pi^2_{s,t}}(s\sim T+s-t,x)&=\inf_{\pi^2_{s,t}\in \Pi_m^2}J^{\pi^1,\pi^2}(t\sim T,x)\\
		&\geq \inf_{\pi^2\in \Pi_m^2}J^{\pi^1,\pi^2}(t\sim T,x)~\forall\pi^1\in \Pi^1_m.
		\end{align*}
		
	Similarly, we can show $J_{*}(t\sim T,x)\geq J_{*}(s\sim T+s-t,x)$ and $J_{*}(t\sim T,x)\leq J_{*}(s\sim T+s-t,x)$. So $J_{*}(t\sim T,x)= J_{*}(s\sim T+s-t,x)$.
		Now since $c(t,x,a,b)\geq 0$ on $K$, by (\ref{eq 4.4a}) and $t>s$, we have $J_{*}(t\sim T,x)=J_{*}(s\sim T+s-t,x)\leq J_{*}(s\sim T,x)$. By this and $J_{*}(t\sim T,x)=J_{*}(t,x)$, we have $J_{*}(t,x)\leq J_{*}(s,x)$. Now using this and by Corollary \ref{corr 3.1}, we say $\varphi(t,x)$ is also decreasing in $t$. This completes the proof.
			\end{enumerate}
			\end{proof}
	\end{enumerate}
\end{proposition}
 \begin{thm}\label{thm 4.1}
 Under Assumptions \ref{assm 3.1}, \ref{assm 3.2} and \ref{assm 4.1}, if in addition $c(t,x,a,b)\geq 0$ and $g(t,x)\geq 0$ for all $(t,x,a,b)\in K$, then there exists a unique $\varphi\in \mathbb{B}^{ac}_{V,V_1}([0,T]\times S)$ and some pair of strategies $(\pi^{*1},\pi^{*2})\in \Pi^1_{m}\times \Pi^2_{m}$ satisfying the equations (\ref{eq 3.21}), (\ref{eq 3.22}) and (\ref{eq 3.22a}).
  		\begin{proof}  		 First observe that $1\leq e^{2(T+1) c(t,x,a,b)}\leq M_2 V(\phi(x,T-t))$ and $1\leq e^{2(T+1) g(t,x)}\leq M_2 V(\phi(x,T-t))$. For each integer $n\geq 1$, $x\in S$, $t\in [0,T]$, define $S_n:=\{x\in S|V(x)\leq n\}$, $A_n(t,x):=A(t,x)$ and $B_n(t,x):=B(t,x).$ 
  		For each $(t,x,a,b)\in K_n:=K$,
  			define 	\begin{align}
  		q_{n}(dy|t,x,a,b)	:=\left\{ \begin{array}{ll}	
  		q(dy|t,x,a,b)~\text{ if}~x\in S_n,\\
  		0~\text{ if }~x\notin S_n,\label{eq 4.5}
  		\end{array}\right.
  		\end{align}
  		\begin{align}
  		c_n^{+}(t,x,a,b):=\left\{ \begin{array}{ll}	
  		c(t,x,a,b)\wedge \text{ min}\{n,\frac{1}{(T+1)}\ln\sqrt{M_2V(\phi(x,T-t))}\}~\text{ if}~x\in S_n,\\
  		0~\text{ if }~x\notin S_n.\label{eq 4.6}
  		\end{array}\right.
  		\end{align}
  		and 
  		\begin{align}
  	g_n^{+}(t,x):=\left\{ \begin{array}{ll}	
  	g(t,x)\wedge \text{ min}\{n,\frac{1}{(T+1)}\ln\sqrt{M_2V(\phi(x,T-t))}\}~\text{ if}~x\in S_n,\\
  	0~\text{ if }~x\notin S_n.\label{eq 4.7}
  	\end{array}\right.
  	\end{align}
 By (\ref{eq 4.5}), it is obvious that $q_{n}(dy|t,x,a,b)$ is transition rates on $S$ satisfying conservative and stable conditions. 
  	  	Now	we consider the sequence of bounded cost rates PDMGs models $\mathscr{M}^{+}_n:=\{S,A, (A_n(t,x),x\in S),B, (B_n(t,x),x\in S),q_n, \phi(x,t), c^{+}_n,g^{+}_n\}$ and for any pair of Markov strategies $(\pi^1,\pi^2)\in \Pi_m^1\times\Pi_m^2$, define the value function $$J_n(t,x):=\sup_{\pi^1\in \Pi_m^1}\inf_{\pi^2\in\Pi_m^2}J^{\pi^1,\pi^2}_n(t,x).$$ Then by Proposition \ref{prop 4.1}, for each $n \geq 1$, we get a unique $\varphi_n$ in $\mathbb{B}^{ac}_{V,V_1}([0,T]\times S)$ and $(\pi^{*1}_n,\pi^{*2}_n)\in \Pi_{m}^1\times \Pi_{m}^2$ satisfying 
  		\interdisplaylinepenalty=0
  		\begin{align}
  		&-L^\phi\varphi_n(t,x)\nonumber\\
  		&=\sup_{\mu\in P(A(t,x))}\inf_{\nu\in P(B(t,x))}\biggl\{\lambda c^{+}_n(t,x,\mu,\nu)\varphi_n(t,x)+\int_{S}\varphi_n(t,y)q_n(dy|t,x,\mu,\nu)\biggr\}\nonumber\\
  		&=\inf_{\nu\in P(B(t,x))}\sup_{\mu\in P(A(t,x))}\biggl\{\lambda c^{+}_n(t,x,\mu,\nu)\varphi_n(t,x)+\int_{S}\varphi_n(t,y)q_n(dy|t,x,\mu,\nu)\biggr\}\nonumber\\
  		&=\inf_{\nu\in P(B(t,x))}\biggl\{\lambda c^{+}_n(t,x,\pi^{*1}_n(\cdot|x,t),\nu)\varphi_n(t,x)+\int_{S}\varphi_n(t,y)q_n(dy|t,\pi^{*1}_n(\cdot|x,t),\nu)\biggr\}\nonumber\\
  		&=\sup_{\mu\in P(A(t,x))}\biggl\{\lambda c^{+}_n(t,x,\mu,\pi^{*2}_n(\cdot|x,t))\varphi_n(t,x)+\int_{S}\varphi_n(t,y)q_n(dy|t,x,\mu,\pi^{*2}_n(\cdot|x,t))\biggr\}\nonumber\\
  		&\varphi_n(t,x)=e^{\lambda g_n(T,x)}~~~s\in [0,T],~x\in S.\label{eq 4.8}
  		\end{align}
  		Now, we have $e^{2(T+1) c^{+}_n(t,x,a,b)}\leq M_2 V(\phi(x,T-t))$ and $e^{2(T+1) g^{+}_n(t,x)}\leq M_2 V(\phi(x,T-t))$. Hence by Lemma \ref{lemm 3.1}, Proposition \ref{prop 4.1} and (\ref{eq 4.8}), we have 
  		\begin{align}
  	e^{-\lambda L_2 V(\phi(x,T-t))}\leq \varphi_{n}(t,x)=\sup_{\pi^1\in \Pi_m^1}J^{\pi^1,\pi^{*2}_n}_n(t,x)\leq L_2 V(\phi(x,T-t))~ \forall~ n\geq 1.\label{eq 4.9}
  		\end{align}
  		 Moreover, since $\varphi_{n}(t,x)\geq 0$ and $c^{+}_n(t,x,a,b)\geq c^{+}_{n-1}(t,x,a,b)$ for all $(t,x,a,b)\in K$, by (\ref{eq 4.5}), (\ref{eq 4.6}), (\ref{eq 4.8}) and Proposition \ref{prop 4.2}, we have, for all $x\in S$ and a.e. $t$, we have 
  		\begin{align}
  		\left\{ \begin{array}{llll}&L^\phi\varphi_n(t,x)+\displaystyle\biggl[\lambda c^{+}_{n-1}(t,x,\mu,\pi^{*2}_n(\cdot|x,t))\varphi_n(t,x)
  		+\int_{S}\varphi_n(t,y)q_{n-1}(dy|t,x,\mu,\pi^{*2}_n(\cdot|x,t))\biggr]\\&\leq 0~\quad\quad\text{if}~x\in S_{n-1}\label{eq 4.10}
  		\end{array}\right.
  		\end{align}
  		and 
  		\begin{align}
  		\left\{ \begin{array}{llll}&L^\phi\varphi_n(t,x)+\displaystyle\biggl[\lambda c^{+}_{n-1}(t,x,\mu,\pi^{*2}_n(\cdot|x,t))\varphi_{n}(t,x)+\int_{S}\varphi_{n}(t,y)q_{n-1}(dy|t,x,\mu,\pi^{*2}_n(\cdot|x,t))\biggr]\\
  		&=L^\phi\varphi_n(t,x)\leq 0~\quad\quad\text{if}~x\notin S_{n-1}.
  		\end{array}\right.\label{eq 4.11}
  		\end{align}
  		So, for any $\pi^1\in \Pi_m^1$, by Feynman-Kac formula, we get
  		\begin{align*}
  		J^{\pi^1,\pi^{*2}_n}_{n-1}(t,x)\leq \varphi_{n}(t,x).
  		\end{align*}
  		Since $\pi^1\in \Pi_m^1$ is arbitrary
  		\begin{align*}
  		\sup_{\pi^1\in \Pi_m^1}	J^{\pi^1,\pi^{*2}_n}_{n-1}(t,x)\leq \varphi_{n}(t,x).
  		\end{align*}
  		Hence
  			\begin{align}
  	\inf_{\pi^2\in \Pi_m^2}	\sup_{\pi^1\in \Pi_m^1}	J^{\pi^1,\pi^2}_{n-1}(t,x)\leq  \varphi_{n}(t,x).\label{eq 4.12}
  		\end{align}
  		Also, using (\ref{eq 4.8}) and Feynman-Kac formula, (similar proof as in Corollary \eqref{corr 3.1}), we have
  	\begin{align}
  	\sup_{\pi^1\in \Pi_m^1}\inf_{\pi^2\in \Pi_m^2}J^{\pi^1,\pi^2}_{n-1}(t,x)=\inf_{\pi^2\in\Pi_m^2}	\sup_{\pi^1\in \Pi_m^1}J^{\pi^1,\pi^2}_{n-1}(t,x)= \varphi_{n-1}(t,x).\label{eq 4.13}
  	\end{align}
  	From (\ref{eq 4.12}) and (\ref{eq 4.13}), we have $\varphi_{n-1}(t,x)\leq  \varphi_{n}(t,x)$, that is the sequence $\{\varphi_{n},n\geq 1\}$ is nondecreasing in $n\geq 1$. Also, since  $\varphi_{n}$ has an upper bound, $\lim_{n\rightarrow \infty}\varphi_{n}$ exists. Let 
\begin{align}
\lim_{n\rightarrow \infty}\varphi_{n}(t,x):=\varphi(t,x).\label{eq 4.14}
\end{align}
  		Let
  	\begin{align*}
  	&H_n(t,x):= \sup_{\mu\in P(A_n(t,x))}\inf_{\nu\in P(B_n(t,x))}\biggl\{\lambda c^{+}_n(t,x,\mu,\nu)\varphi_n(t,x)
  	+\int_{S}\varphi_n(t,y)q_n(dy|t,x,\mu,\nu)\biggr\},\\
  	&\quad \forall (t,x)\in [0,T]\times S.
  	\end{align*}
  	Then, by Fan's minimax theorem \cite{Fan}, we have
  	\begin{align*}
  	&H_n(t,x):= \inf_{\nu\in P(B_n(t,x))}\sup_{\mu\in P(A_n(t,x))}\biggl\{\lambda c^{+}_n(t,x,\mu,\nu)\varphi_n(t,x)
  	+\int_{S}\varphi_n(t,y)q_n(dy|t,x,\mu,\nu)\biggr\},\\
  	&\quad \forall (t,x)\in [0,T]\times S.
  	\end{align*}
  	
  	Then, by Assumptions \ref{assm 3.1} and \ref{assm 3.2} and the fact that $\lambda\leq 1$, we get the following result
  	\interdisplaylinepenalty=0
  	\begin{align}
  	|H_n(t,x)|&\nonumber\\
  	&\leq \sup_{\mu\in P(A_n(t,x))}\sup_{\nu\in P(B_n(t,x))}\biggl\{\lambda|c^{+}_n(t,x,\mu,\nu)\varphi_n(t,x)|
  	+\int_{S}|\varphi_n(t,y)q_n(dy|t,x,\mu,\nu)|\biggr\}\nonumber\\
  	& \leq L_2\sup_{\mu\in P(A_n(t,x))}\sup_{\nu\in P(B_n(t,x))}\biggl\{ M_2V(\phi(x,T-t)) V(\phi(x,T-t))+\int_{S}|q_n(dy|t,x,\mu,\nu)|V(\phi(y,T-t))\biggr\}\nonumber\\
  	&\leq L_2\biggl( M_2V^2(\phi(x,T-t))+(b_1+\rho_1V(\phi(x,T-t)))+2q^{*}(x)V(\phi(x,T-t))\biggr)\nonumber\\
  	&\leq L_2\biggl(M_2 V^2(\phi(x,T-t))+(b_1+\rho_1)V^2(\phi(x,T-t))+2M_1V^2(\phi(x,T-t))\biggr)\nonumber\\
  	&\leq L_2M_3V_1(\phi(x,T-t))(M_2+b_1+\rho_1+2M_1)\nonumber\\
  	&=:R(x),~\forall (t,x)\in [0,T]\times S.\label{eq 4.15}
  	\end{align} 		
  		Now, we prove that for each fixed $x\in S$, the equicontinuty of the family $\{\varphi_n(\cdot,x)\}_{n\geq 1}$ on $[0,T]$.
  		So, fix arbitrarily some $x\in S$ and $s,s_0\in [0,T]$,
  		\begin{align}
  		&|\varphi_n(s,x)-\varphi_n(s_0,x)|\nonumber\\
  		&=\biggl|\int_{s}^{T}H_n(t,x)dt-\int_{s_0}^{T}H_n(t,x)dt\biggr|\nonumber\\
  		&\leq R(x)|s-s_0|,\quad \forall n\geq 1.\label{eq 4.16}
  		\end{align}
Hence for each $x\in S$, $\varphi_{n}(\cdot,x)$ is Lipschitz continuous in $t\in [0, T]$. Also, $\varphi_{n}(t,x)$ is increasing as $n\rightarrow\infty$ for any $(t,x)\in [0,T]\times S$, therefore there exists a function $\varphi$ on $[0,T]\times S$ that is continuous with respect to $t\in [0,T]$, such that along a subsequence $\{n_k\}$, we have $\lim_{k\rightarrow\infty}\varphi_{n_k}(\cdot,x)=\varphi(\cdot,x)$.
Now by Lemma \ref{lemm 3.1}, we get
  		\begin{equation}
  		|\varphi(t,x)|\leq L_2V(\phi(x,T-t))\quad \forall t\in [0,T].\label{eq 4.17}
  		\end{equation}
  		
  		Let
  		\begin{align*}
  		&H(t,x):=\sup_{\mu\in P(A(t,x))}\inf_{\nu\in P(B(t,x))}\biggl\{\lambda c(t,x,\mu,\nu)\varphi(t,x)+\int_{S}\varphi(t,y)q(dy|t,x,\mu,\nu)\biggr\},\\
  		&\quad \forall (t,x)\in [0,T]\times S.
  		\end{align*}
  		Then by Fan's minimax theorem, \cite{Fan}, we have
  			\begin{align*}
  			&H(t,x):=\inf_{\nu\in P(B(t,x))}\sup_{\mu\in P(A(t,x))}\biggl\{\lambda c(t,x,\mu,\nu)\varphi(t,x)+\int_{S}\varphi(t,y)q(dy|t,x,\mu,\nu)\biggr\},\\
  			&\quad \forall (t,x)\in [0,T]\times S.
  			\end{align*}
  		We next show that for each fixed $x\in S$ and $t\in [0,T]$,
  		$\lim_{k\rightarrow \infty}H_{n_k}(t,x)=H(t,x)$.
  		Now, since $P(A(t,x))$ is a compact set, under Lemma \ref{lemm 4.1}, by measurable selection theorem in \cite{N}, there exists a sequence of strategies $\mu_{n_k}^{*}\in P(A(t,x))$ such that 
  		\begin{align}
  		H_{n_k}(t,x):= \inf_{\nu\in P(B(t,x))} \biggl\{\lambda c^{+}_{n_k}(t,x,\mu_{n_k}^{*},\nu)\varphi_{n_k}(t,x)
  		+\int_{S}\varphi_{n_k}(t,y)q_{n_k}(dy|t,x,\mu_{n_k}^{*},\nu)\biggr\}.\label{eq 4.18}
  		\end{align}
  		Let $(t,x)\in [0,T]\times S$ be arbitrarily fixed. Since, $P(A(t,x))$ is compact, by taking subsequences if necessary, we assume without loss of generality that $\limsup_{k\rightarrow \infty}H_{n_k}(t,x)=\lim_{k\rightarrow \infty}H_{n_k}(t,x)$ and $\mu^{*}_{n_k}\rightarrow \mu^{*}$ as $k\rightarrow\infty$  for some $\mu^{*}\in P(A(t,x))$. Taking $k\rightarrow\infty$, by Lemma 8.3.7 in Hernandez-Lerma and Lassere (1999) in \cite{HL2}, for arbitrarily fixed $\nu \in P(B(t,x))$, we have	
  		\begin{align*}
  		\limsup_{k\rightarrow\infty}H_{n_k}(t,x)
  		&\leq \limsup_{k\rightarrow \infty}\biggl\{\lambda c^{+}_{n_k}(t,x,\mu_{n_k}^{*},\nu)\varphi_{n_k}(t,x)+\int_{S}\varphi_{n_k}(t,y)q_{n_k}(dy|t,x,\mu_{n_k}^{*},\nu)\biggr\}\\
  		&\leq \biggl\{\lambda c(t,x,\mu^{*},\nu)\varphi(t,x)+\int_{S}\varphi(t,y)q(dy|t,t,x,\mu^{*},\nu)\biggr\}.
  		\end{align*}
  		Since $\nu \in P(B(t,x))$ is arbitrary,
  		\begin{align}
  		\limsup_{k\rightarrow \infty}H_{n_k}(t,x)
  		&\leq \inf_{\nu\in P(B(t,x))}\biggl\{\lambda c(t,x,\mu^{*},\nu)\varphi(t,x)+\int_{S}\varphi(t,y)q(dy|t,x,\mu^{*},\nu)\biggr\}\nonumber\\
  		&\leq \sup_{\mu\in P(A(t,x))}\inf_{\nu\in P(B(t,x))}\biggl\{\lambda c(t,x,\mu,\nu)\varphi(t,x)+\int_{S}\varphi(t,y)q(dy|t,x,\mu,\nu)\biggr\}.\label{eq 4.19}
  		\end{align}
  		Using analogous arguments we can show that
  		\begin{align}
  		\liminf_{k\rightarrow \infty}H_{n_k}(t,x)
  		&\geq \inf_{\nu\in P(B(t,x))}\sup_{\mu\in P(A(t,x))}\biggl\{\lambda c(t,x,\mu,\nu)\varphi(t,x)
  		+\int_{S}\varphi(t,y)q(dy|t,x,\mu,\nu)\biggr\}.\label{eq 4.20}
  		\end{align}
  		We know that $\liminf_{k\rightarrow \infty}H_{n_k}(t,x)\leq \limsup_{k\rightarrow\infty} H_{n_k}(t,x)$.
  		So, by (\ref{eq 4.19}) and (\ref{eq 4.20}), we get 
  		  		\interdisplaylinepenalty=0
  	\begin{align}
  		\lim_{k\rightarrow \infty}H_{n_k}(t,x)
  		&=\inf_{\nu\in P(B(t,x))}\sup_{\mu\in P(A(t,x))}\biggl\{\lambda c(t,x,\mu,\nu)\varphi(t,x)
  		+\int_{S}\varphi(t,y)q(dy|t,x,\mu,\nu)\biggr\}\nonumber\\
  		&=\sup_{\mu\in P(A(t,x))}\inf_{\nu\in P(B(t,x))}\biggl\{\lambda c(t,x,\mu,\nu)\varphi(t,x)
  		+\int_{S}\varphi(t,y)q(dy|t,x,\mu,\nu)\biggr\}.\label{eq 4.21}
  		\end{align}
  		That is $\lim_{k\rightarrow \infty}H_{n_k}(t,x)=H(t,x)$ for all $(t,x)\in [0,T]\times S$.\\
  		Since $\lim_{k\rightarrow\infty}\varphi_{n_k}(\cdot,x)=\varphi(\cdot,x)$ and $(s,x)\in [0,T]\times S$ was arbitrarily fixed, we see using, (\ref{eq 4.21}) taking limit $n_k\rightarrow\infty$ in (\ref{eq 4.8}), by the dominated convergent theorem (since $|H_{n_k}(t,x)|\leq R(x)$), we say that $\varphi$ satisfies (\ref{eq 3.21}).
  		So, we obtain
  		  		\begin{align}
  	&	\varphi(s,x)-e^{g(T,\phi(x,T-s))}\nonumber\\
  		&=\int_{s}^{T}\sup_{\mu\in P(A(z,\phi(x,z-s)))}\inf_{\nu\in P(B(z,\phi(x,z-s)))}\biggl\{\lambda c(z,\phi(x,z-s),\mu,\nu)\varphi(z,\phi(x,z-s))\nonumber\\
  		&+\int_{S}\varphi(z,y)q(dy|z,\phi(x,z-s),\mu,\nu)\biggr\}dz\nonumber\\
  		&=\int_{s}^{T}\inf_{\nu\in P(B(z,\phi(x,z-s)))}\sup_{\mu\in P(A(z,\phi(x,z-s)))}\biggl\{\lambda c(z,\phi(x,z-s),\mu,\nu)\varphi(z,\phi(x,z-s))\nonumber\\
  		&+\int_{S}\varphi(z,y)q(dy|z,\phi(x,z-s),\mu,\nu)\biggr\}dz\nonumber\\
  		&\varphi_n(s,x)=e^{\lambda g(T,x)}~~~s\in [0,T],~x\in S.\label{eq 4.8n}
  		\end{align}
  		Now, by (\ref{eq 3.21}) and the fundamental theorem of Lebesgue integral calculus  \cite[Theorem 4.4.1]{AL}, $\varphi(\cdot,x)$ is absolutely continuous on $[0,T]$ and by the same argument as in (\ref{eq 4.15}) gives
  		\begin{align*}
  	|L^\phi\varphi(t,x)|=|H(t,x)|\leq R(x),\quad  \forall (t,x)\in [0,T]\times S.
  		\end{align*}
  		
  		Therefore, we see that $\varphi\in \mathbb{B}^{ac}_{V,V_1}([0,T]\times S)$. Furthermore, measurable selection theorem in \cite{N} implies that there exists a pair of strategies $(\pi^{*1},\pi^{*2})\in \Pi_{m}^1\times \Pi_{m}^2$ satisfying (\ref{eq 3.21}).\\
  		Finally, we have to verify the uniqueness of the solution $\varphi$ of the optimality equation (\ref{eq 3.21}). Let $\varphi\in \mathbb{B}^{ac}_{V,V_1}([0,T]\times S)$ be an arbitrarily fixed solution to (\ref{eq 3.21}).
  	Now by Proposition \ref{prop 4.1}, we see $\varphi(t,x)=\sup_{\pi^1\in \Pi_m^1}\inf_{\pi^2\in \Pi_m^2}J^{\pi^1,\pi^2}(t,x)$. So, $\varphi$ is the unique solution to (\ref{eq 3.21}) satisfying (\ref{eq 3.22a}), out of $\varphi\in \mathbb{B}^{ac}_{V,V_1}([0,T]\times S)$.	
 	\end{proof}
  \end{thm}
 
 The main optimal result is the following one from which we know the existence of saddle-point equilibrium and the existence of the value of the game.
 \begin{thm}\label{thm 4.2}
	Suppose Assumptions \ref{assm 3.1}, \ref{assm 3.2} and \ref{assm 4.1} are satisfied. Then the following assertions hold.
	\begin{enumerate}
		\item [(a)]  There exists a unique $\tilde{\varphi}\in \mathbb{B}^{ac}_{V,V_1}([0,T]\times S)$ and some pair of strategies $(\pi^{*1},\pi^{*2})\in \Pi^1_{m}\times \Pi^2_{m}$ satisfying the equations (\ref{eq 3.21}), (\ref{eq 3.22}) and (\ref{eq 3.22a}).
		\item [(b)] The pair of strategies, $(\pi^{*1},\pi^{*2})\in \Pi^1_{m}\times \Pi^2_{m}$ in part (a) is a saddle-point equilibrium.
	\end{enumerate}
\begin{proof}
\begin{itemize}
	\item [(a)]	For each $n\geq 1$, define $c_n$ and $g_n$ on $K$ as: 
	$$c_n(t,x,a,b):=\max\{-n,c(t,x,a,b)\},~g_n(t,x):=\max\{-n,g(t,x)\}$$ for each $(t,x,a,b)\in K$.
	Then $\lim_{n\rightarrow \infty}c_n(t,x,a,b)=c(t,x,a,b)$ and $\lim_{n\rightarrow \infty}g_n(t,x)=g(t,x)$. Define $r_n(t,x,a,b):=c_n(t,x,a,b)+n$ and $\tilde{g}_n(t,x):=g_n(t,x)+n$. So, $r_n(t,x,a,b)\geq 0$ and $\tilde{g}_n(t,x)\geq 0$ for each $n\geq 1$ and $(t,x,a,b)\in K$. Now by Assumption \ref{assm 3.1}, we have 
	\begin{align}
	-\frac{1}{T+1}\ln\sqrt{M_2V(\phi(x,T-t))}&\leq \max\{-n,-\frac{1}{T+1}\ln\sqrt{M_2V(\phi(x,T-t))}\}\nonumber\\
	&\leq c_n(t,x,a,b)\leq \frac{1}{T+1}\ln\sqrt{M_2V(\phi(x,T-t))}\label{eq 4.22}
	\end{align}
	and
		\begin{align}
	-\frac{1}{T+1}\ln\sqrt{M_2V(\phi(x,T-t))}&\leq \max\{-n,-\frac{1}{T+1}\ln\sqrt{M_2V(\phi(x,T-t))}\}\nonumber\\
	&\leq g_n(t,x)\leq \frac{1}{T+1}\ln\sqrt{M_2V(\phi(x,T-t))}.\label{eq 4.23}
	\end{align}
	So, we have $e^{2(T+1)r_n(t,x,a,b)}\leq e^{2(T+1)n}M_2V(\phi(x,T-t))$ and $e^{2(T+1)\tilde{g}_n(t,x)}\leq e^{2(T+1)n}M_2V(\phi(x,T-t))$ for all $n\geq 1$ and $(t,x,a,b)\in K$.
	Define a new model $\mathscr{R}_n:=\{S,(A_n(t,x),x\in S),(B_n(t,x),x\in S),q,\phi(x,t),r_n,\tilde{g}_n\}$.
	Now for any real-valued measurable functions $\tilde{\psi}$ and $\tilde{\phi}$ defined on $K$and $[0,T]\times S$, respectively, define 
	\begin{align}
	&\tilde{J}^{\tilde{\psi},\tilde{\phi}}(s,x)\nonumber\\&:=\sup_{\pi^1\in \Pi_m^1}\inf_{\pi^2\in \Pi_m^2} {E}^{\pi^1,\pi^2}_{(s,x)}\biggl[exp\biggl(\lambda\int_{s}^{T}\int_{B}\int_{A}\tilde{\psi}(t,\xi_t,a,b)\pi^1(da|\xi_t,t)\pi^2(db|\xi_t,t)dt+\lambda \tilde{\phi}(T,\xi_T)\biggr)\biggr]\label{eq 4.24}
	\end{align}
	assuming that the integral exists.
	Now since $r_n\geq 0$, $\tilde{g}_n\geq 0$ and all Assumptions hold for the model  $\mathscr{R}_n$, by Theorem \ref{thm 4.1}, we have
	\begin{align}
	&-L^\phi \tilde{J}^{r_n,\tilde{g}_n}(s,x)\nonumber\\
	&=\sup_{\mu\in P(A(t,x))}\inf_{\nu\in P(B(t,x))}\biggl\{\lambda r_n(t,x,\mu,\nu)	\tilde{J}^{r_n,\tilde{g}_n}(t,x)+\int_{S}	\tilde{J}^{r_n,\tilde{g}_n}(t,y)q(dy|t,x,\mu,\nu)\biggr\}\nonumber\\
	&=\inf_{\nu\in P(B(t,x))}\sup_{\mu\in P(A(t,x))}\biggl\{\lambda r_n(t,x,\mu,\nu)\tilde{J}^{r_n,\tilde{g}_n}(t,x)+\int_{S}	\tilde{J}^{r_n,\tilde{g}_n}(t,y)q(dy|t,x,\mu,\nu)\biggr\}\nonumber\\
	&=\inf_{\nu\in P(B(t,x))}\biggl\{\lambda r_n(t,x,\pi^{*1}(\cdot|x,t),\nu)	\tilde{J}^{r_n,\tilde{g}_n}(t,x)+\int_{S}	\tilde{J}^{r_n,\tilde{g}_n}(t,y)q(dy|t,x,\pi^{*1}(\cdot|x,t),\nu)\biggr\}\nonumber\\
	&=\sup_{\mu\in P(A(t,x))}\biggl\{\lambda r_n(t,x,\mu,\pi^{*2}(\cdot|x,t))	\tilde{J}^{r_n,\tilde{g}_n}(t,x)+\int_{S}	\tilde{J}^{r_n,\tilde{g}_n}(t,y)q(dy|t,x,\mu,\pi^{*2}(\cdot|x,t))\biggr\}\nonumber\\
	&\tilde{J}^{r_n,\tilde{g}_n}(T,x)=e^{\tilde{g}_n(T,x)}~~~s\in [0,T],~x\in S.\label{eq  4.25}
	\end{align}
	 Now 
	$$\tilde{J}^{r_n,\tilde{g}_n}(s,x)=\tilde{J}^{c_n+n,g_n+n}(s,x)=\tilde{J}^{c_n,g_n}(s,x)e^{\lambda(T-s+1)n}.$$
	So, using (\ref{eq  4.25}), we can write for a.e. $s\in [0,T]$, 
		\begin{align}
	&-L^\phi \tilde{J}^{c_n,g_n}(s,x)\nonumber\\
	&=\sup_{\mu\in P(A(s,x))}\inf_{\nu\in P(B(s,x))}\biggl\{\lambda c_n(s,x,\mu,\nu)\tilde{J}^{c_n,g_n}(s,x)+\int_{S}\tilde{J}^{c_n,g_n}(s,y)q(dy|s,x,\mu,\nu)\biggr\}\nonumber\\
	&=\inf_{\nu\in P(B(s,x))}\sup_{\mu\in P(A(s,x))}\biggl\{\lambda c_n(s,x,\mu,\nu)\tilde{J}^{c_n,g_n}(s,x)+\int_{S}\tilde{J}^{c_n,g_n}(s,y)q(dy|s,x,\mu,\nu)\biggr\}\nonumber\\
	&\tilde{J}^{c_n,{g}_n}(T,x)=e^{{g}_n(T,x)}.\label{eq 4.27}
	\end{align}
	Now by (\ref{eq 4.27}) and Lemma \ref{lemm 3.1}, we have
	\begin{align}
	|\tilde{J}^{c_n,g_n}(t,x)|\leq L_2V(\phi(x,T-t)),~~n\geq 1.\label{eq 4.28}
	\end{align}
	Now since $c_n(t,x,a,b)$  and $g_n(t,x)$ are decreasing in $n\geq 1$, so the corresponding value function $\tilde{J}^{c_n,g_n}(t,x)$ is also decreasing in $n$. Also by Lemma \ref{lemm 3.1}, we know that $\tilde{J}^{c_n,g_n}$ has a lower bound. So, $\lim_{n\rightarrow \infty}\tilde{J}^{c_n,g_n}(t,x)$ exists. Let $\lim_{n\rightarrow \infty}\tilde{J}^{c_n,g_n}(t,x)=:\tilde{\varphi}(t,x)$ for each $(t,x)\in [0,T]\times S$. Then by the same arguments as Theorem 4.1, with $\varphi_{n}(t,x)$ replaced with $\tilde{J}^{c_n,g_n}(t,x)$ here, by (\ref{eq 4.27}), (\ref{eq 4.28}), Assumptions \ref{assm 3.1}, and \ref{assm 3.2}, we see that (a) is true. 
	\item [(b)] Now by measurable selection theorem, there exists a pair of strategies $(\pi^{*1},\pi^{*2})\in \Pi_m^1\times \Pi_m^2$ for which (\ref{eq 3.21}) holds. So, by Proposition \ref{prop 4.1}, we get
		\begin{align*}  		 
	& \sup_{\pi^1\in\Pi^1}\inf_{\pi^2\in \Pi^2}J^{\pi^{1},\pi^{2}}(0,x)=\inf_{\pi^2\in \Pi^2}\sup_{\pi^1\in\Pi^1}J^{\pi^{1},\pi^{2}}(0,x)=\sup_{\pi^1\in\Pi^1}J^{\pi^{1},\pi^{*2}}(0,x) \\
	&=\inf_{\pi^2\in\Pi^2}J^{\pi^{*1},{\pi}^{2}}(0,x)=\tilde{\varphi}(0,x)
	\end{align*}
and
	\begin{align*}  		 
	& \sup_{\pi^1\in\Pi_m^1}\inf_{\pi^2\in \Pi_m^2}J^{\pi^{1},\pi^{2}}(t,x)=\inf_{\pi^2\in \Pi_m^2}\sup_{\pi^1\in\Pi_m^1}J^{\pi^{1},\pi^{2}}(t,x)=\sup_{\pi^1\in\Pi_m^1}J^{\pi^{1},\pi^{*2}}(t,x) \\
	&=\inf_{\pi^2\in\Pi_m^2}J^{\pi^{*1},{\pi}^{2}}(t,x)=\tilde{\varphi}(t,x).
	\end{align*}
From these together with the equations (\ref{eq 2.5}), (\ref{eq 2.6}), (\ref{eq 3.1}) and (\ref{eq 3.2}), we get
		\begin{align*}
	& \sup_{\pi^1\in\Pi^1}\inf_{\pi^2\in \Pi^2}\mathscr{J}^{\pi^{1},\pi^{2}}(0,x)
	=\inf_{\pi^2\in \Pi^2}\sup_{\pi^1\in\Pi^1}\mathscr{J}^{\pi^{1},\pi^{2}}(0,x)
	=\sup_{\pi^1\in\Pi^1}\mathscr{J}^{\pi^{1},\pi^{*2}}(0,x)\\
	&=\inf_{\pi^2\in\Pi^2}\mathscr{J}^{\pi^{*1},{\pi}^{2}}(0,x)
	= \frac{1}{\lambda}\ln\tilde{\varphi}(0,x)
	\end{align*}
		and
	\begin{align*}
	& \sup_{\pi^1\in\Pi_m^1}\inf_{\pi^2\in \Pi_m^2}\mathscr{J}^{\pi^{1},\pi^{2}}(t,x)
	=\inf_{\pi^2\in \Pi_m^2}\sup_{\pi^1\in\Pi_m^1}\mathscr{J}^{\pi^{1},\pi^{2}}(t,x)
	=\sup_{\pi^1\in\Pi_m^1}\mathscr{J}^{\pi^{1},\pi^{*2}}(t,x)\\
	&=\inf_{\pi^2\in\Pi_m^2}\mathscr{J}^{\pi^{*1},{\pi}^{2}}(t,x)
	= \frac{1}{\lambda}\ln\tilde{\varphi}(t,x).
	\end{align*}
	Hence
		\begin{align*}
	&\sup_{\pi^1\in\Pi^1}\inf_{\pi^2\in \Pi^2}\mathscr{J}^{\pi^{1},\pi^{2}}(0,x) =\sup_{\pi^1\in\Pi_m^1}\inf_{\pi^2\in \Pi_m^2}\mathscr{J}^{\pi^{1},\pi^{2}}(0,x)\\	&=\inf_{\pi^2\in \Pi^2}\sup_{\pi^1\in\Pi^1}\mathscr{J}^{\pi^{1},\pi^{2}}(0,x)
	=\inf_{\pi^2\in \Pi_m^2}\sup_{\pi^1\in\Pi_m^1}\mathscr{J}^{\pi^{1},\pi^{2}}(0,x)\\
	&=\sup_{\pi^1\in\Pi^1}\mathscr{J}^{\pi^{1},\pi^{*2}}(0,x)
	=\sup_{\pi^1\in\Pi_m^1}\mathscr{J}^{\pi^{1},\pi^{*2}}(0,x)\\
	&=\inf_{\pi^2\in\Pi^2}\mathscr{J}^{\pi^{*1},{\pi}^{2}}(0,x)
	=\inf_{\pi^2\in\Pi_m^2}\mathscr{J}^{\pi^{*1},{\pi}^{2}}(0,x)\\
	&= \frac{1}{\lambda}\ln\tilde{\varphi}(0,x).
	\end{align*}
	So, the pair of strategies $(\pi^{*1},\pi^{*2})\in \Pi_{m}^1\times \Pi_{m}^2$ is a saddle-point equilibrium. Therefore, we complete the proof of the theorem.
	 \begin{remark}
		Note that $({\pi}^{*1},{\pi}^{*2})\in \Pi_m^1\times \Pi_m^2$ as defined in theorem above is a saddle-point equilibrium for the cost criterion (\ref{eq 2.5}) and  we have  $U(x)=L(x)=\mathscr{J}^{*}(0,x)=\frac{1}{\lambda}\log \tilde{\varphi}(0,x)$.  Thus the value of the game exists.
	\end{remark}  
\end{itemize}
		\end{proof}
\end{thm}
  
\section{Competing Interests}
The authors do not have any relevant financial or non-financial competing interests.

\bibliographystyle{elsarticle-num}

\end{document}